\documentclass[12pt]{amsart}
\usepackage[utf8]{inputenc}
\usepackage{setspace}
\usepackage[pagewise]{lineno}
\usepackage{setspace}
\usepackage[margin=1in]{geometry}
\usepackage[english]{babel}
\usepackage{amsthm}
\usepackage{amsmath}
\usepackage{amsfonts}
\usepackage{amssymb}
\usepackage{amsthm, thmtools}
\usepackage{
hyperref,
cleveref,
}
\declaretheorem[name=Theorem,
refname={Theorem,Theorems},
Refname={Theorem,Theorems},
numberwithin=section]{theorem}
\declaretheorem[name=Lemma,
refname={Lemma,Lemmas},
Refname={Lemma,Lemmas},
sibling=theorem,
style=definition]{lemma}
\declaretheorem[name=Corollary,
refname={Corollary,Corollaries},
Refname={Corollary,Corollaries},
sibling=theorem,
style=definition]{col}
\declaretheorem[name=Example,
refname={Example,Examples},
Refname={Example,Examples},
sibling=theorem,
style=definition]{eg}
\declaretheorem[name=Proposition,
refname={Proposition,Propositions},
Refname={Proposition,Propositions},
sibling=theorem,
style=definition]{propn}
\declaretheorem[name=Definition,
refname={Definition,Definitions},
Refname={Definition,Definitions},
sibling=theorem,
style=definition]{defn}

\declaretheorem[name=Remark,
refname={Remark,Remarks},
Refname={Remark,Remarks},
sibling=theorem]{rmk}
\declaretheorem[name=Notation,
refname={Notation,Notations},
Refname={Notation,Notations},
sibling=theorem]{ntn}

\newcommand{\R}{\mathbb{R}}
\newcommand{\sR}{\mathcal{R}}

\newcommand{\N}{\mathbb{N}}

\newcommand{\Q}{\mathbb{Q}}
\newcommand{\supp}{\text{supp}}
\newcommand{\bkl}{\text{ \textbackslash \ }}

\title{On a Lebesgue-like integral over the Levi-Civita field $\sR$}
\author{Mateo Restrepo Borrero}
\address{Department of Mathematics, Universidad Nacional de Colombia, Colombia}
\email{marestrepob@unal.edu.co}
\author{Khodr Shamseddine}
\address{Department of Physics and Astronomy and Department of Mathematics, University of Manitoba, Winnipeg, Manitoba
R3T 2N2, Canada}
\email{khodr.shamseddine@umanitoba.ca}
\date{June, 2025}
\begin{document}

\maketitle

\begin{abstract}
The Levi-Civita field $\sR$ is the smallest non-Archimedean ordered field extension of the real numbers that is real closed and Cauchy complete in the topology induced by the order. In this paper we develop a new theory of integration over $\sR$ that generalizes previous work done in the subject while circumventing the fact that not every bounded subset of $\sR$ admits either an infimum or a supremum. We define a new family of measurable functions and a new integral over measurable subsets of $\sR$ that satisfies some very important results analogous to those of the Lebesgue and Riemann integrals for real-valued functions. In particular, we show that the family of measurable functions forms an algebra that is closed under taking absolute values, that the integral is linear, countably additive and monotone, that the integral of a non-negative function $f:A\to\sR$ is zero if and only if $f=0$ almost everywhere in $A$ and that the analogous versions for the uniform convergence theorem and the fundamental theorem of calculus hold true.
\end{abstract}
\pagebreak
\section{Introduction}\label{S1}
\subsection{The Levi-Civita field $\sR$}
We recall that $\sR$ is defined as the set of all functions $f:\Q\to\R$ such that $\supp (f)$ is left-finite, where $\supp(f):=\{q\in\Q\mid f(q)\neq0\}$. That is, for every $q\in\Q$, there are only finitely many $q'\in\supp(f)\cap(-\infty,q)$. We denote the value of $f$ at $q$ by $f[q]$; and we define
\[
\lambda(f):=\begin{cases}
    \min\left(\supp(f)\right)&\text{if }f\neq 0\\
    \infty &\text{if }f=0.
\end{cases}
\]

We equip $\sR$ with the following operations:
\begin{align*}
    (f+g)[q]:&=f[q]+g[q]\\
    (f\cdot g)[q]:&=\sum_{{\tiny \begin{array}{cc} q_1+q_2=q\\
    q_1\in\supp(f),q_2\in\supp(g)
    \end{array}}}f[q_1]\cdot g[q_2]
\end{align*}
where the sum on the right hand side is well-defined due to the left-finiteness of the supports. We usually denote the elements of $\sR$ by the variables $x$, $y$, $z$, etc. instead of $f$, $g$, $h$, etc., which we will use to denote functions on $\sR$.

It can be shown \cite{shamseddinephd} that $\sR$ forms a field under the addition and multiplication defined above. Moreover, if we define the relations $<_\sR$ and $\leq_\sR$ by
\begin{eqnarray*}
f<_\sR g &\iff& \left[f\neq g\text{ and }(g-f)[\lambda(g-f)]>0\right]\text{ and }\\
f\le_\sR g&\iff& f<_\sR g  \text{ or }f=g
\end{eqnarray*}
then \((\sR,+,\cdot,\leq_\sR)\) is a totally ordered field which we will denote simply by $\sR$.

The standard topology on $\sR$ is the order topology and it is not difficult to check that it is equivalent to the topology induced by the following ultrametric absolute value:
\[
|x|_u:=\begin{cases}
      0 & \text{if }x=0 \\
      e^{-\lambda(x)} &\text{otherwise}.
   \end{cases}
\]
It is well-known \cite{NASurvey22} that $(\sR,|\cdot|_u)$ is the smallest Cauchy-complete non-Archimedean valued field extension of the real numbers that is real closed. Moreover, $\sR$ is the smallest real closed
 field extension of $\R$ that is Cauchy-complete with respect to a non-Archimedean
 order. Thus, it is of special interest to study.
 
The field $\mathcal{R}$ and its complex counterpart $\mathcal{C}:=\sR\oplus i\sR$ \cite{ComplexLC24} are just special cases of the class of fields discussed in \cite{schikhofbook}. For a general overview of the algebraic properties of formal power series fields in general, we refer the reader to the comprehensive overview by Ribenboim \cite{ribenboim92}, and for an overview of the related valuation theory to the books by Krull \cite{krull32}, Schikhof \cite{schikhofbook} and Alling \cite{allingbook}. A thorough and complete treatment of ordered structures can also be found in \cite{priessbook}.

\begin{rmk}
    Due to the non-Archimedean nature of $\sR$, the order of limits and sums can be interchanged much more easily than in $\R$. For further details we refer the reader to \cite{rspsio00}.
\end{rmk}
\subsection{Measure and integration over $\sR$}
A first attempt at introducing a measure in $\sR$ was made in \cite{RSMNi02} (which we refer to as the S-measure). Unfortunately, the family of S-measurable sets is a bit too small to develop the theory much further. In \cite{LMeasure22}, we developed a new family of measurable sets and a new measure over the family that strictly contains the S-measurable sets while assigning them their S-measure. This new family of sets forms an algebra that's closed under some countable unions and intersections and the measure satisfies some very useful continuity results that we shall employ later. We recall here the main definitions from \cite{LMeasure22} and some key properties.
\begin{defn}[Outer measure]
    Let $A\subset\sR$ be given. Then we say that $A$ is outer measurable if
    $$\inf \left\{\sum\limits_{n=1}^\infty l(S_n): S_n\text{'s are intervals and }A\subseteq \bigcup\limits_{n=1}^\infty S_n \right\}$$
    exists in $\sR$. If so, we call that number the outer measure of $A$ and denote it by $m_u(A)$.
\end{defn}
\begin{defn}[Measurable sets]
   Let $A\subset\mathcal{R}$ be an outer measurable set. Then we say that $A$ is measurable if for every other outer measurable set $B\subset\mathcal{R}$ both $A\cap B$ and $A^c\cap B$ are outer measurable and
    $$m_u(B)=m_u(A\cap B)+m_u(A^c\cap B).$$
    In this case, we define the measure of $A$ to be $m(A):=m_u(A)$. The measure satisfies the following properties:
\end{defn}
\begin{enumerate}
    \item\label{propunint}\label{propunlmeas} Let $A,B\subseteq\sR$ be measurable. Then $A\cap B,A\cup B$ and $A\cap B^c$ are measurable. Moreover,
    \[
    m(A\cup B)=m(A)+m(B)-m(A\cap B).
    \]
    \item\label{propMu0}Let $C\subset\mathcal{R}$ be outer measurable with $m_u(C)=0$. Then $C$ is measurable with $m(C)=0$.
    \item For each $n\in\N$ let $A_n\subset\mathcal{R}$ be measurable.
    \begin{enumerate}
        \item Suppose $\lim\limits_{N\to\infty}m\left(\bigcup\limits_{n=1}^N A_n\right)$ exists in $\sR$. Then, $\bigcup\limits_{n=1}^\infty A_n$ is measurable and has measure
     \[
     m\left(\bigcup\limits_{n=1}^\infty A_n\right)=\lim\limits_{N\to\infty}m\left(\bigcup\limits_{n=1}^N A_n\right).
     \]
     \item Suppose $\lim\limits_{N\to\infty}m\left(\bigcap\limits_{n=1}^N A_n\right)$ exists in $\sR$. Then, $\bigcap\limits_{n=1}^\infty A_n$ is measurable and has measure
     \[
     m\left(\bigcap\limits_{n=1}^\infty A_n\right)=\lim\limits_{N\to\infty}m\left(\bigcap\limits_{n=1}^N A_n\right).
     \]
    \end{enumerate}
\end{enumerate}
Despite not satisfying some of these properties, the S-measure proved strong enough to develop a first attempt at integration over $\sR$. We'll introduce the reader to an adjusted definition and some basic results that will be needed later.
\begin{defn}
    Given $r\in\R$, we define the mapping $\lVert\cdot\rVert_r:\sR\to\R$ as follows:
    \[
    \lVert x\rVert_r:=\max\{|x[q]|:q\in\Q\text{ and }q\leq r\}.
    \]
\end{defn}
We note that $\lVert\cdot\rVert_r$ is a semi-norm and that the countable family of semi-norms ${\{\lVert\cdot\rVert_t:t\in\Q\}}$ induces a metric topology on $\sR$ that is strictly weaker than the standard topology induced by the order or the ultrametric absolute value \cite{reftoplc}.
\begin{defn}[Weak convergence]
    A sequence $(a_n)_{n\in\N}$ in $\sR$ is said to be weakly convergent to some $a\in\sR$ if for all $\epsilon>0$ in $\R$, there exists some $N\in\N$ such that for $n\geq N$ $\lVert a_n-a\rVert_{1/\epsilon}<\epsilon$.
\end{defn}
\begin{defn}[$\sR$-analytic function]
    Let $a<b$ in $\sR$ be given and $f:[a,b]\to\sR$. We say that $f$ is $\sR$-analytic on $[a,b]$ if for all $x\in[a,b]$, there exists some $\delta>0$ in $\sR$ with $\lambda(\delta)=\lambda(b-a)$ and a sequence $(a_n(x))_{n\in\N}$ in $\sR$ such that $\bigcup\limits_{n=1}^\infty\supp (a_n(x))$ is left-finite and, under weak convergence, $f(y)=\sum\limits_{n=1}^\infty a_n(x)(y-x)^n$ for all $y\in(x-\delta,x+\delta)\cap[a,b]$. 
\end{defn}
The following are some properties that $\sR$-analytic functions satisfy. For further details on $\sR$-analytic functions we refer the reader to \cite{surveyRanalytic} and the references therein.
\begin{enumerate}
    \item The family of $\sR$-analytic functions over a closed interval forms an algebra that contains the identity function.
    \item Suppose $f$ is $\sR$-analytic on $[a,b]$. Then $f$ is differentiable on $[a,b]$ and its derivative is also $\sR$-analytic on $[a,b]$. Moreover, if $f'(x) = 0$ for all $x\in[a,b]$, then $f$ is constant.
    \item Suppose $f$ is $\sR$-analytic on $[a,b]$. Then, there exists an $\sR$-analytic function $F$ on $[a,b]$ such that $F'=f$. Moreover, the anti-derivative $F$ is unique up to a constant.
    \item Suppose $f$ is $\sR$-analytic on $[a,b]$. Then $f$ satisfies the intermediate value theorem, the extreme value theorem and the mean value theorem. Moreover, if $f$ is not constant and $y$ is in between $f(a)$ and $f(b)$, then $f(x)=y$ only finitely many times for $x\in[a,b]$.
\end{enumerate}
\begin{defn}[Integral of an $\sR$-analytic function]
    Let $a<b$ in $\sR$, let $f:[a,b]\to\sR$ be $\sR$-analytic on $[a,b]$ and let $F$ be an $\sR$-analytic anti-derivative of $f$ in $[a,b]$. Then, the integral of $f$ over $[a,b]$ is the $\sR$ number
    \[
    \int_{[a,b]}f\,dx:=F(b)-F(a).
    \]
\end{defn}
\begin{defn}[S-measurable functions]
    Let $(J_n)$ be a collection of intervals with mutually disjoint interiors and $\lim\limits_{n\to\infty}l(J_n)=0$, let $A:=\bigcup\limits_{n=1}^\infty J_n$ and let $f:A\to\sR$ be bounded. We say that $f$ is S-measurable on $A$ if for all $\epsilon>0$ in $\sR$, there exists a sequence of closed intervals $(I_n)$ whose interiors do not overlap such that $I_n\subseteq A$ for all $n$, the sum $\sum\limits_{n=1}^\infty l(I_n)$ converges in $\sR$, $\sum\limits_{n=1}^\infty l(J_n)-\sum\limits_{n=1}^\infty l(I_n)<\epsilon$ and $f$ is $\sR$-analytic on $I_n$ for each $n$.
\end{defn}
\begin{defn}[Integral of an S-measurable function]
    Let $(J_n)$ be a collection of intervals with mutually disjoint interiors and $\lim\limits_{n\to\infty}l(J_n)=0$, let $A:=\bigcup\limits_{n=1}^\infty J_n$ and let $f:A\to\sR$ be S-measurable. Then, the integral of $f$ over $A$ is defined as
    \[
    \int_Af\,dx=\lim\limits_{\sum\limits_{n=1}^\infty l(I_n)\to m(A)}\left(\sum\limits_{n=1}^\infty\int_{I_n}f\,dx\right)
    \]
    where $\bigcup\limits_{n=1}^\infty I_n\subseteq A$, $(I_n)$ is a collection of closed intervals whose interiors do not overlap and $f$ is $\sR$-analytic on $I_n$ for each $n$.
\end{defn}
That the integral is well-defined can be found in \cite{RSMNi02} and \cite{rsint2}. There, the authors also show that the family of S-measurable functions forms an algebra and that the integral is linear, additive and monotone.
\section{Simple functions}\label{S3}
\begin{defn}
    Let $A\subseteq \sR$ be measurable and let $f:A\to \sR$ be bounded. We say that $f$ is a simple function on $A$ if for all $\epsilon>0$ in $\sR$ there exists some collection of closed intervals $\{I_n\}_{n=1}^\infty$ whose interiors do not overlap (which we will call interval cover) and a bounded function $\hat{f}:\bigcup\limits_{n=1}^\infty I_n\to\sR$ such that $A\subseteq \bigcup\limits_{n=1}^\infty I_n$, ${\sum\limits_{n=1}^\infty l(I_n)-m(A)<\epsilon}$, $\hat{f}$ is $\sR$-analytic on each $I_n$ and for all $x\in A$, $f(x)=\hat{f}(x)$. We call such a function a simple extension of $f$ over $\bigcup\limits_{n=1}^\infty I_n$.
\end{defn}
\begin{ntn}
    We say that a sequence $\{I^k_n\}_{n=1}^\infty$ of interval covers of $A$ converges to $A$ if 
    \[
    {\lim\limits_{k\to\infty}\sum\limits_{n=1}^\infty l(I^k_n)=m(A)}.
    \]
\end{ntn}
\begin{propn}
     Let $A\subseteq \sR$ be measurable and let $f:A\to \sR$ be simple. Then, for any other measurable set $B\subseteq A$, the function $\left.f\right|_B:B\to\sR$ is simple.
\end{propn}
\begin{proof}
    Let $\epsilon>0$ in $\sR$ be given, let $\{I_n\}_{n=1}^\infty$ be an interval cover of $A$ and let $\hat{f}:\bigcup\limits_{n=1}^\infty I_n\to\sR$ be a simple extension of $f$. Since $B$ is measurable, we can find some interval cover $\{J_m\}_{m=1}^\infty$ of $B$ such that ${\sum\limits_{m=1}^\infty l(J_m)-m(B)<\epsilon}$. Thus, the collection $\{K_{n,m}\}_{n,m=1}^\infty$ given by
    \[
    K_{n,m}:=I_n\cap J_m
    \]
    is the desired interval cover of $B$.
\end{proof}
\begin{propn}\label{abs-is-simple}
    Let $A\subseteq \sR$ be measurable and let $f:A\to \sR$ be simple. Then, the function $|f|:A\to\sR$ is simple.
\end{propn}
\begin{proof}
    Let $\epsilon>0$ in $\sR$ be given. By definition, there exists some interval cover $\{I_n\}_{n=1}^\infty$ of $A$ and a simple extension $\hat{f}:\bigcup\limits_{n=1}^\infty I_n\to\sR$ of $f$ such that $\sum\limits_{n=1}^\infty l(I_n)-m(A)<\epsilon$. Now, since $\hat{f}$ is $\sR$-analytic on each of the closed intervals $I_n$, it only changes signs finitely many times there. Thus, for each $n\in\N$, there exists some finite collection of closed intervals $\{J_{n,m}\}_{m=1}^{N_{n}}$ whose interiors do not overlap such that $I_n=\bigcup\limits_{m=1}^{N_{n}}J_{n,m}$ and $\left|\hat{f}\right|$ is $\sR$-analytic on each $J_{n,m}$. Thus, $\bigcup\limits_{n=1}^\infty \{J_{n,m}\}_{m=1}^{N_{n}}$ is an interval cover of $A$ and $\left|\hat{f}\right|:\bigcup\limits_{n=1}^\infty\bigcup\limits_{m=1}^{N_n}J_{n,m}\to\sR$ is a simple extension of $|f|$, with
    \[
    \sum\limits_{n=1}^\infty \sum\limits_{m=1}^{N_n} l(J_{n,m})-m(A)=\sum\limits_{n=1}^\infty l(I_n)-m(A)<\epsilon.
    \]
\end{proof}
\begin{col}
     Let $A\subseteq \sR$ be measurable and let $f,g:A\to R$ be simple. Then, the functions
     \[
     \min\{f,g\},\max\{f,g\}:A\to\sR
     \]
     are simple.
\end{col}
\begin{propn}\label{simple-measurable}
    Let $A\subseteq \sR$ be measurable and let $f:A\to\sR$ be simple. Then, if $I$ is an interval, then $f^{-1}(I)$ is measurable.
\end{propn}
\begin{proof}
    Let $s$ be $\sR$-analytic on $[a,b]$. Then, $s^{-1}(I)$ is a finite collection of intervals and therefore measurable. Now, given that $f$ is simple, there exist some interval cover $\{I_n\}_{n=1}^\infty$ of $A$ and a simple extension $\hat{f}:\bigcup\limits_{n=1}^\infty I_n\to\sR$ of $f$ such that $\sum\limits_{n=1}^\infty l(I_n)$ converges in $\sR$. We define $\hat{f_n}:=\left.\hat{f}\right|_{I_n}$ and note that
    \[
    \hat{f}^{-1}(I)=\bigcup\limits_{n=1}^\infty\left(\hat{f}_n\right)^{-1}(I).
    \]
    Since for each $n\in\N$, $\left(\hat{f}_n\right)^{-1}(I)$ is measurable and $\left(\hat{f}_n\right)^{-1}(I)\subseteq I_n$, then
    \[
    m\left(\left(\hat{f}_n\right)^{-1}(I)\right)\leq l(I_n)\underset{n\to\infty}{\to}0.
    \]
    Thus, $\bigcup\limits_{n=1}^\infty\left(\hat{f}_n\right)^{-1}(I)$ is measurable. It follows that the set
    \[
    f^{-1}(I)=A\cap\hat{f}^{-1}(I)
    \]
    is measurable. The result then follows.
\end{proof}
\subsection{Integral of a simple function}
In the following, we prepare for the definition for the integral of a simple function $f$ over a measurable set $A$.
\begin{rmk}
    Let $A$ be measurable, let $\epsilon>0$ in $\sR$ be given and let $A\subseteq\bigcup\limits_{n=1}^\infty I_n $, such that for each $n\in\N$, $I_n$ is an interval, $A\cap I_n\neq\phi$ and $\sum\limits_{n=1}^\infty l(I_n)-m(A)<\epsilon$. Then, for all $n\in\N$ and for any $x\in I_n$, there exists some $y\in A\cap I_n$ such that $|x-y|<\epsilon$.
\end{rmk}
\begin{propn}\label{ext-ind}
    Let $A\subseteq \sR$ be measurable and let $f:A\to\sR$ be a simple function. Let $\{I_n\}_{n=1}^\infty$ be an interval cover of $A$ and let ${\hat{f},\hat{g}:\bigcup\limits_{n=1}^\infty I_n\to \sR}$ be simple extensions of $f$. Then, for all $\epsilon>0$ in $\sR$, there exists an interval cover $\{J_n\}_{n=1}^\infty$ of $A$ such that $\bigcup\limits_{n=1}^\infty J_n\subseteq \bigcup\limits_{n=1}^\infty I_n$ and ${\left|\sum\limits_{n=1}^\infty\int_{J_n}(\hat{f}-\hat{g})\,dx\right|<\epsilon}$.
\end{propn}
\begin{proof}
    Let $\epsilon>0$ in $\sR$ be given. We define $A_n:=A\cap I_n$. Since for all $n\in\N$, both $\hat{f}$ and $\hat{g}$ are uniformly continuous on $I_n$, then there exists some $\delta_n$ such that if $x,y\in I_n$ and $|x-y|<\delta_n$, then $|\hat{f}(x)-\hat{f}(y)|<\dfrac{d^n\epsilon}{l(I_n)}$ and $|\hat{g}(x)-\hat{g}(y)|<\dfrac{d^n\epsilon}{l(I_n)}$.\\

    Now, since each $A_n$ is measurable, for each $n\in\N$ we can find an interval cover $\{J_{n,m}\}_{m=1}^\infty$ of $A_n$, such that $J_{n,m}\subseteq I_n$, $A_n\cap J_{n,m}\neq \phi$ for each $m\in\N$ and $\sum\limits_{m=1}^\infty l(J_{n,m})-m(A_n)<\delta_n$. It follows that for each $x\in J_{n,m}$, there exists some $y\in J_{n,m}\cap A_n$ such that $|x-y|<\delta_n$, and thus, for $x\in J_{n,m}$ we have
    \begin{align*}
        |\hat{f}(x)-\hat{g}(x)|&\leq |\hat{f}(x)-f(y)|+|f(y)-\hat{g}(x)|\\
        &=|\hat{f}(x)-\hat{f}(y)|+|\hat{g}(y)-\hat{g}(x)|\\
        &<\dfrac{2d^n\epsilon}{l(I_n)}.
    \end{align*}
    Hence, the interval cover $\bigcup\limits_{n=1}^\infty\{J_{n,m}\}_{m=1}^\infty$ of $A$ satisfies that
    \begin{align*}
        \left|\sum\limits_{n=1}^\infty\sum\limits_{m=1}^\infty\int_{J_{n,m}}(\hat{f}-\hat{g})\,dx\right|&=\left|\sum\limits_{n=1}^\infty\int_{\bigcup\limits_{m=1}^\infty J_{n,m}}(\hat{f}-\hat{g})\,dx\right|\\
        &\leq \sum\limits_{n=1}^\infty\left|\int_{\bigcup\limits_{m=1}^\infty J_{n,m}}(\hat{f}-\hat{g})\,dx\right|\\
        &\leq \sum\limits_{n=1}^\infty m\left(\bigcup\limits_{m=1}^\infty J_{n,m}\right)\dfrac{2d^n\epsilon}{l(I_n)}\\
        &\leq \sum\limits_{n=1}^\infty 2d^n\epsilon\\
        &<\epsilon,
    \end{align*}
    which proves the result.
\end{proof}
\begin{propn}\label{path-ind}
    Let $A\subseteq \sR$ be measurable, let $f:A\to\sR$ be simple and let $\left(\{I_n^k\}_{n=1}^\infty\right)_{k=1}^\infty$ be a sequence of interval coverings of $A$ that converges to $A$ with $\bigcup\limits_{n=1}^\infty I^k_n\subseteq\bigcup\limits_{n=1}^\infty I^1_n$ for all $k$. Let $\hat{f}:\bigcup\limits_{n=1}^\infty I^1_n\to \sR$ be a simple extension of $f$. Then, the sequence $(i_k)_{k=1}^\infty$ given by
    \[
    i_k:=\sum\limits_{n=1}^\infty\int_{I_n^k}\hat{f}\,dx
    \]
    converges.
\end{propn}
\begin{proof}
    Let $\epsilon>0$ in $\sR$ be given. Let $M>0$ in $\sR$ be a bound of $\hat{f}$ and let $K\in \N$ be such that if $k>K$, then
    \[
    \sum\limits_{n=1}^{\infty}l(I_n^k)-m(A)<\dfrac{d\epsilon}{M},
    \]
    Then, for each $k>K$ there exists some $N_k\in\N$ such that
    \[
        \sum\limits_{n=N_k+1}^\infty\left|\int_{I_n^k}\hat{f}\,dx\right|+\sum\limits_{n=N_k+1}^\infty\left|\int_{I_n^{k+1}}\hat{f}\,dx\right|<d\epsilon
    \]
    and
    \[
        \sum\limits_{n,m=1}^{\infty}l(I_n^k\cap I_m^{k+1})-\sum\limits_{n,m=1}^{N_k}l(I_n^k\cap I_m^{k+1})<\dfrac{d\epsilon}{M}.
    \]
    Let $A_k:=\bigcup\limits_{n=1}^{N_k} I^k_n$ and $A_{k+1}:=\bigcup\limits_{n=1}^{N_k} I^{k+1}_n$. We note that
    \[
    \int_{A_k}\hat{f}\,dx-\int_{A_{k+1}}\hat{f}\,dx=\int_{A_k\bkl A_{k+1}}\hat{f}\,dx-\int_{A_{k+1}\bkl A_k}\hat{f}\,dx.
    \]
    Thus,
    \begin{align*}
        \left|\int_{A_k}\hat{f}\,dx-\int_{A_{k+1}}\hat{f}\,dx\right|&\leq\left|\int_{A_k\bkl A_{k+1}}\hat{f}\,dx\right|+\left|\int_{A_{k+1}\bkl A_k}\hat{f}\,dx\right|\\
        &\leq M\left(m\left(A_{k+1}\bkl A_k\right)+m\left(A_{k+1}\bkl A_k\right)\right)\\
        &=M(m(A_k)-m(A_k\cap A_{k+1})+m(A_{k+1})-m(A_k\cap A_{k+1}))\\
        &=M\left(\sum\limits_{n=1}^{N_k} l(I^k_n)-\sum\limits_{n,m=1}^{N_k} l(I^k_n\cap I^{k+1}_m)+\sum\limits_{n=1}^{N_k} l(I^{k+1}_n)-\sum\limits_{n,m=1}^{N_k} l(I^k_n\cap I^{k+1}_m)\right)\\
        &\leq M\left(\sum\limits_{n=1}^{\infty} l(I^k_n)-\sum\limits_{n,m=1}^{N_k} l(I^k_n\cap I^{k+1}_m)+\sum\limits_{n=1}^{\infty} l(I^{k+1}_n)-\sum\limits_{n,m=1}^{N_k} l(I^k_n\cap I^{k+1}_m)\right)\\
        &< M\left(\sum\limits_{n=1}^{\infty} l(I^k_n)-\sum\limits_{n,m=1}^{\infty} l(I^k_n\cap I^{k+1}_m)+\sum\limits_{n=1}^{\infty} l(I^{k+1}_n)-\sum\limits_{n,m=1}^{\infty} l(I^k_n\cap I^{k+1}_m)+\dfrac{2d\epsilon}{M}\right)\\
        &\leq M\left(\sum\limits_{n=1}^{\infty} l(I^k_n)-m(A)+\sum\limits_{n=1}^{\infty} l(I^{k+1}_n)-m(A)+\dfrac{2d\epsilon}{M}\right)\\
        &<4d\epsilon.
    \end{align*}
    It follows that
    \begin{align*}
        |i_k-i_{k+1}|&=\left|\sum\limits_{n=1}^\infty\int_{I_n^k}\hat{f}\,dx-\sum\limits_{n=1}^\infty\int_{I_n^{k+1}}\hat{f}\,dx\right|\\
        &\leq \left|\sum\limits_{n=1}^{N_k}\int_{I_n^k}\hat{f}\,dx-\sum\limits_{n=1}^{N_k}\int_{I_n^{k+1}}\hat{f}\,dx\right|+\sum\limits_{n=N_k+1}^\infty\left|\int_{I_n^k}\hat{f}\,dx\right|+\sum\limits_{n=N_k+1}^\infty\left|\int_{I_n^{k+1}}\hat{f}\,dx\right|\\
        &=\left|\int_{A_k}\hat{f}\,dx-\int_{A_{k+1}}\hat{f}\,dx\right|+\sum\limits_{n=N_k+1}^\infty\left|\int_{I_n^k}\hat{f}\,dx\right|+\sum\limits_{n=N_k+1}^\infty\left|\int_{I_n^{k+1}}\hat{f}\,dx\right|\\
        &<5d\epsilon\\
        &<\epsilon.
    \end{align*}
    Thus, the sequence $(i_k)_{k=1}^\infty$ is Cauchy and therefore convergent.
\end{proof}
Put together, \cref{ext-ind} and \cref{path-ind} prove the following theorem:
\begin{theorem}
    Let $A\subseteq \sR$ be measurable and let $f:A\to \sR$ be simple. Then, the limit
    \[
    \lim\limits_{k\to\infty} \sum\limits_{n=1}^\infty \int_{I_n^k}\hat{f}\,dx
    \]
    exists and is both independent of the choice of the simple extension $\hat{f}$ of $f$ and of the sequence of interval covers $\{I_n^k\}$ that converges to $A$.
\end{theorem}
\begin{defn}
    Let $A\subseteq \sR$ be measurable and let $f:A\to \sR$ be simple. We define
    \[
    \int_A f\,dx:=\lim\limits_{k\to\infty} \sum\limits_{n=1}^\infty \int_{I_n^k}\hat{f}\,dx
    \]
    where $\hat{f}$ is a simple extension of $f$ and $\{I_n^k\}$ is a sequence of interval coverings of $A$ that converges to $A$.
\end{defn}
\subsection{Properties of the integral of a simple function}
We now show that the integral of a simple function satisfies some nice properties akin of those of the Riemann and Lebesgue Integrals.
\begin{propn}
     Let $A\subseteq \sR$ be measurable and $\alpha\in\sR$. Then, the constant function $\alpha:A\to \sR$ is simple and
    \[
    \int_A \alpha \,dx=\alpha m(A).
    \]
\end{propn}
\begin{propn}[Linearity]
    Let $A\subseteq \sR$ be measurable, let $\alpha\in\sR$ and let $f,g:A\to R$ be simple. Then, the functions $\alpha f+g$ and $f\cdot g$ are simple on $A$. Moreover
    \[
     \int_A (\alpha f+g)\,dx=\alpha \int_A f\,dx+ \int_A g\,dx.
    \]
\end{propn}
\begin{proof}
    It's easy to check that if $\hat{f}$ and $\hat{g}$ are simple extensions of $f$ and $g$ respectively, then $\alpha\hat{f}+\hat{g}$ and $\hat{f}\cdot\hat{g}$ are simple extensions of $\alpha f+g$ and $f\cdot g$, respectively. Moreover,
    \begin{align*}
        \int_A (\alpha f+g)\,dx&=\lim\limits_{k\to\infty} \sum\limits_{n=1}^\infty \int_{I_n^k}(\alpha\hat{f}+\hat{g})\,dx\\
        &=\alpha\lim\limits_{k\to\infty} \sum\limits_{n=1}^\infty \int_{I_n^k}\hat{f}\,dx+\lim\limits_{k\to\infty}\sum\limits_{n=1}^\infty \int_{I_n^k}\hat{g}\,dx\\
        &=\alpha \int_A f\,dx+ \int_A g\,dx.
    \end{align*}
\end{proof}
\begin{propn}
    Let $A\subseteq \sR$ be measurable and let $f,g:A\to \sR$ be simple such that $f\leq g$. Then
    \[
    \int_A f\,dx\leq \int_A g\,dx.
    \]
\end{propn}
\begin{proof}
    It suffices to show that if $f\geq 0$ on $A$, then $\int_A f\,dx\geq 0$.\\

    Let $\{I_n^k\}$ be a sequence of interval covers that converges to $A$ and let $\hat{f}:\bigcup\limits_{n=1}^\infty I_n^1\to\sR$ be a simple extension of $f$. Since $\hat{f}$ is $\sR$-analytic on each $I_n^k$, then, for each $n,k\in\N$, there exists some finite collection of closed intervals $\{J^k_{n,m}\}_{m=1}^{N_{n,k}}$ whose interiors do not overlap such that $I_n^k=\bigcup\limits_{m=1}^{N_{n,k}}J^k_{n,m}$ and $\left|\hat{f}\right|$ is $\sR$-analytic on each $J^k_{n,m}$. Now, the sequence of interval covers $\bigcup\limits_{n=1}^\infty \{J^k_{n,m}\}_{m=1}^{N_{n,k}}$ converges to $A$ and $\left|\hat{f}\right|:\bigcup\limits_{n=1}^\infty\bigcup\limits_{m=1}^{N_{n,k}}J^k_{n,m}\to\sR$ is a simple extension of $f$. Since $\left|\hat{f}\right|\geq 0$, then $\int_{J^k_{n,m}}\left|\hat{f}\right|\,dx\geq 0$ and therefore
    \[
    \int_A f\,dx=\lim\limits_{k\to\infty}\sum\limits_{n=1}^\infty\sum\limits_{m=1}^{N_{n,k}}\int_{J^k_{n,m}}\left|\hat{f}\right|\,dx\geq 0.
    \]
\end{proof}
\begin{col}
    Let $A\subseteq \sR$ be measurable and $f:A\to R$ be simple. Then, if $M$ is a bound of $f$ then
    \[
    \left|\int_A f\,dx\right|\leq Mm(A).
    \]
\end{col}
\begin{propn}[Additivity]
    Let $A,B\subseteq \sR$ be measurable with $B\subseteq A$ and let $f:A\to\sR$ be simple. Then
    \[
    \int_Af\,dx=\int_Bf\,dx+\int_{A\bkl B}f\,dx.
    \]
\end{propn}
\begin{proof}
    Let $\{I_n^k\}$ and $\{J_m^k\}$ be a sequence of interval coverings converging to $B$ and $A\bkl B$ respectively. We note that the union $\bigcup\limits_{n=1}^\infty I_n^k\cup\bigcup\limits_{m=1}^\infty J_m^k$ can be rewritten as $\bigcup\limits_{n=1}^\infty I_n^k\cup\bigcup\limits_{m=1}^\infty J_m^k=\bigcup\limits_{r=1}^\infty R_r^k$ where $\{R_r^k\}$ is an interval cover of $A$. It follows, from properties of the measure in $\sR$, that
    \[
    \lim\limits_{k\to\infty} \sum\limits_{r=1}^\infty l(R^k_r)=m(A)
    \]
    and
    \[
    \lim\limits_{k\to\infty} \sum\limits_{n,m=1}^\infty l(I_n^k\cap J_m^k)=0.
    \]
    Now, let $\{I_n\}$ be an interval cover of $A$ and let $\hat{f}:\bigcup\limits_{n=1}^\infty I_n\to\sR$ be a simple extension of $f$. We may assume, without loss of generality, that for each $k\in\N$, $\bigcup\limits_{n=1}^\infty I_n^k,\bigcup\limits_{m=1}^\infty J_m^k\subseteq \bigcup\limits_{n=1}^\infty I_n$. Then, $\hat{f}$ is a simple extension of $f$ over $\bigcup\limits_{n=1}^\infty I_n^k$ and $\bigcup\limits_{m=1}^\infty J_m^k$ and
    \[
    \int_{\bigcup\limits_{n=1}^\infty I_n^k}\hat{f}\,dx+\int_{\bigcup\limits_{m=1}^\infty J_m^k}\hat{f}\,dx-\int_{\bigcup\limits_{r=1}^\infty R_r^k}\hat{f}\,dx=\int_{\bigcup\limits_{n,m=1}^\infty I_n^k\cap J_m^k}\hat{f}\,dx.
    \]
    Taking the limit as $k\to\infty$ yields
    \[
    \int_Bf\,dx+\int_{A\bkl B}f\,dx-\int_A f\,dx=\int_\phi f\,dx=0,
    \]
    which proves the result.
\end{proof}
\begin{col}
    Let $A,B\subseteq \sR$ be measurable and let $f:A\cup B\to\sR$ be simple. Then
    \[
    \int_Af\,dx+\int_Bf\,dx=\int_{A\cup B}f\,dx-\int_{A\cap B}f\,dx.
    \]
\end{col}
\begin{col}
    Let $A\subseteq\sR$ be measurable and let $f:A\to \sR$ be simple. If $A=\bigcup\limits_{n=1}^NA_n$ where each $A_n$ is measurable and $A_i\cap A_j=\phi$ whenever $i\neq j$, then
    \[
    \int_Af\,dx=\sum\limits_{n=1}^N\int_{A_n}f\,dx.
    \]
\end{col}
\begin{col}
    Let $A\subseteq\sR$ be measurable and let $f:A\to \sR$ be simple. If $A=\bigcup\limits_{n=1}^\infty A_n$ where each $A_n$ is measurable, $A_i\cap A_j=\phi$ whenever $i\neq j$ and $\lim\limits_{n\to\infty}m(A_n)=0$, then
    \[
    \int_Af\,dx=\sum\limits_{n=1}^\infty \int_{A_n}f\,dx.
    \]
\end{col}
\begin{proof}
    Let $N\in\N$ be given and let $R_N:=\bigcup\limits_{n=N+1}^\infty A_n$. We note that $\lim\limits_{N\to\infty}m(R_N)=0$ and thus $\lim\limits_{N\to\infty}\int_{R_N}f\,dx=0$. It follows that
    \[
    \int_Af\,dx-\sum\limits_{n=1}^N \int_{A_n}f\,dx=\int_{R_N}f\,dx.
    \]
    Taking the limit as $N\to\infty$ proves the result.
\end{proof}
\section{Measurable functions}\label{S4}
\begin{defn}
    Let $A\subseteq \sR$ be measurable. We say that a collection $\{A_n\}_{n=1}^\infty$ of mutually disjoint measurable sets is a partition of $A$ if $A=\bigcup\limits_{n=1}^\infty A_n$ and $m(A_n)\underset{n\to\infty}{\to}0$.
\end{defn}
\begin{defn}
    Let $A\subseteq \sR$ be measurable and let $\{A_n\}_{n=1}^\infty$ be a partition of $A$. We say that a partition $\{B_m\}_{m=1}^\infty$ of $A$ is a refinement of $\{A_n\}_{n=1}^\infty$ if for all $m\in\N$ there exists some $n\in\N$ such that $B_m\subseteq A_n$.
\end{defn}
\begin{defn}
    Let $A\subseteq \sR$ be measurable and let $\{A_n\}_{n=1}^\infty$ and $\{B_m\}_{m=1}^\infty$ be partitions of $A$. We define the common refinement of $\{A_n\}_{n=1}^\infty$ and $\{B_m\}_{m=1}^\infty$ to be the partition $\{R_{n,m}\}_{n,m=1}^\infty$ defined by $R_{n,m}:=A_n\cap B_m$.
\end{defn}
\begin{defn}
    Let $A\subseteq \sR$ be measurable and let $f:A\to\sR$ be a function. We say that $f$ is locally bounded on $A$ if there exists a partition $\{A_n\}_{n=1}^\infty$ of $A$ such that $f$ is bounded on each $A_n$.
\end{defn}
\begin{defn}
    Let $A\subseteq \sR$ be measurable and let $f:A\to\sR$ be a function. We say that $f$ is measurable if for every $\epsilon>0$ in $\sR$ there exists a partition $\{A_n\}_{n=1}^\infty$ of $A$ and two collections of simple functions $\{i_n:A_n\to\sR\}_{n=1}^\infty$, $\{s_n:A_n\to\sR\}_{n=1}^\infty$ such that $i_n\leq f\leq s_n$ for each $n$, the series $ \sum\limits_{n=1}^\infty \int_{A_n}|s_n|\,dx$ and $ \sum\limits_{n=1}^\infty \int_{A_n}|i_n|\,dx$ both converge in $\sR$ and
    \[
    \sum\limits_{n=1}^\infty \int_{A_n}(s_n-i_n)\,dx<\epsilon.
    \]
\end{defn}
\begin{rmk}
    Note that since $\left|\int_{A_n}i_n\,dx\right|\leq\int_{A_n}|i_n|\,dx$ and $\left|\int_{A_n}s_n\,dx\right|\leq\int_{A_n}|s_n|\,dx$, then, if the series $ \sum\limits_{n=1}^\infty \int_{A_n}|s_n|\,dx$ and $ \sum\limits_{n=1}^\infty \int_{A_n}|i_n|\,dx$ both converge in $\sR$, then, the series $ \sum\limits_{n=1}^\infty \int_{A_n}s_n\,dx$ and $ \sum\limits_{n=1}^\infty \int_{A_n}i_n\,dx$ both converge in $\sR$.
\end{rmk}
\begin{propn}
     Let $A\subseteq \sR$ and $f:A\to\sR$ be measurable. Then $f$ is locally bounded.
\end{propn}
\begin{proof}
    Since $f:A\to\sR$ is measurable, then, for $\epsilon:=1>0$, there exists a partition $\{A_n\}_{n=1}^\infty$ of $A$ and two collections of simple functions $\{i_n:A_n\to\sR\}_{n=1}^\infty$, $\{s_n:A_n\to\sR\}_{n=1}^\infty$ such that $ i_n\leq f\leq s_n$ for each $n$, the series $ \sum\limits_{n=1}^\infty \int_{A_n}|s_n|\,dx$ and $ \sum\limits_{n=1}^\infty \int_{A_n}|i_n|\,dx$ both converge in $\sR$ and
    \[
    \sum\limits_{n=1}^\infty \int_{A_n}(s_n-i_n)\,dx<1.
    \]
    Since for each $n\in\N$, $s_n$ is bounded on $A_n$, it follows that $f$ is bounded on each $A_n$.
\end{proof}
\begin{propn}
    Let $A\subseteq \sR$, let $f,g:A\to\sR$ be measurable and let $\alpha \in\sR$. Then $\alpha f+g:A\to\sR$ is measurable.
\end{propn}
\begin{proof}
    Without loss of generality, we may assume $\alpha>0$.\\

    Let $\epsilon>0$ in $\sR$ be given, let $\{A_n\}_{n=1}^\infty$ be a partition of $A$ and let ${\{i_n:A_n\to\sR\}_{n=1}^\infty}$, $\{s_n:A_n\to\sR\}_{n=1}^\infty$, ${\{l_n:A_n\to\sR\}_{n=1}^\infty}$, $\{u_n:A_n\to\sR\}_{n=1}^\infty$ be collections of simple functions such that $ i_n\leq f\leq s_n$ and $l_n\leq g\leq u_n$ for each $n$, the series $\sum\limits_{n=1}^\infty \int_{A_n}|u_n|\,dx$, $\sum\limits_{n=1}^\infty \int_{A_n}|l_n|\,dx$, $\sum\limits_{n=1}^\infty \int_{A_n}|s_n|\,dx$ and $ \sum\limits_{n=1}^\infty \int_{A_n}|i_n|\,dx$ all converge in $\sR$,
    \[
    \sum\limits_{n=1}^\infty \int_{A_n}(s_n-i_n)\,dx<\dfrac{d\epsilon}{\alpha}
    \]
    and
    \[
    \sum\limits_{n=1}^\infty \int_{A_n}(u_n-l_n)\,dx<d\epsilon.
    \]
    We note that $ \alpha i_n+l_n\leq \alpha f + g \leq\alpha s_n+u_n$ and that the series $\sum\limits_{n=1}^\infty \int_{A_n}|\alpha i_n+l_n|\,dx$ and ${\sum\limits_{n=1}^\infty \int_{A_n}|\alpha s_n+u_n|\,dx}$ both converge in $\sR$.  Furthermore,
    \[
    \sum\limits_{n=1}^\infty \int_{A_n}(\alpha s_n+u_n-(\alpha i_n+l_n))\,dx=\alpha \sum\limits_{n=1}^\infty\int_{A_n}(s_n-i_n)\,dx+ \sum\limits_{n=1}^\infty\int_{A_n}(u_n-l_n)<\alpha\dfrac{d\epsilon}{\alpha}+d\epsilon<\epsilon,
    \]
    thus, $\alpha f+g$ is measurable.
\end{proof}
\begin{propn}\label{rest-is-measurable}
    Let the sets $B\subseteq A\subseteq \sR$ be measurable and let $f:A\to\sR$ be a measurable function. Then, the restriction of $f$ to $B$ is measurable.
\end{propn}
\begin{proof}
    It's enough to notice that if $\{A_n\}_{n=1}^\infty$ is a partition of $A$, then $\{B_n\}_{n=1}^\infty$ given by $B_n:=B\cap A_n$ is a partition of $B$.
\end{proof}
\begin{theorem}\label{R-Criteria}
    Let $A\subseteq \sR$ and $f:A\to\sR$ be measurable. Then, the set
    \begin{align*}
       U_{f,A}:=\left\{\sum\limits_{n=1}^\infty \int_{A_n}s_n\,dx:\right. &\{A_n\}_{n=1}^\infty \text{ is a partition of $A$, $s_n$ is simple in } A_n, \\
       &\left.\text{the series $\sum\limits_{n=1}^\infty \int_{A_n}|s_n|\,dx$ converges in $\sR$ and }f\leq s_n\right\}\
    \end{align*}
    has an infimum. Respectively, the set
    \begin{align*}
       L_{f,A}:=\left\{\sum\limits_{n=1}^\infty \int_{A_n}i_n\,dx:\right. &\{A_n\}_{n=1}^\infty \text{ is a partition of $A$, $i_n$ is simple in } A_n, \\
       &\left.\text{the series $\sum\limits_{n=1}^\infty \int_{A_n}|i_n|\,dx$ converges in $\sR$ and }i_n\leq f\right\}\
    \end{align*}
    has a supremum. Moreover,
    \[
    \inf(U_{f,A})=\sup(L_{f,A}).
    \]
\end{theorem}
\begin{proof}
    We first show that there exists a sequence of partitions $(\{A_n^k\}_{n=1}^\infty)_{k=1}^\infty$ of $A$ and two sequences of collections of simple functions $(\{i_n^k:A_n^k\to\sR\}_{n=1}^\infty)_{k=1}^\infty$, $(\{s_n^k:A_n^k\to\sR\}_{n=1}^\infty)_{k=1}^\infty$ such that $ i^k_n\leq f\leq s^k_n$ for each $n$ and for each $k$, the series $ \sum\limits_{n=1}^\infty \int_{A_n}|s_n^k|\,dx$ and $ \sum\limits_{n=1}^\infty \int_{A_n}|i^k_n|\,dx$ both converge in $\sR$ for each $n$ and the limits
    \[
    \lim\limits_{k\to\infty}\sum\limits_{n=1}^\infty \int_{A_n^k}i^k_n\,dx
    \text{ and }
    \lim\limits_{k\to\infty}\sum\limits_{n=1}^\infty \int_{A_n^k}s^k_n\,dx
    \]
    both exist and are equal.

    We know, by definition, that there exists some partition $\{A^1_n\}_{n=1}^\infty$ of $A$ and two collections of simple functions $\{i_n^1:A_n^1\to\sR\}_{n=1}^\infty$, $\{s_n^1:A_n\to\sR\}_{n=1}^\infty$ such that $ i_n^1\leq f\leq s_n^1$ for each $n$, the series $ \sum\limits_{n=1}^\infty \int_{A_n}|s^1_n|\,dx$ and $ \sum\limits_{n=1}^\infty \int_{A_n}|i_n^1|\,dx$ both converge in $\sR$ and
    \[
    \sum\limits_{n=1}^\infty \int_{A_n}(s_n^1-i_n^1)\,dx<d.
    \]
    Now, given some fixed $k\in \N$, let's suppose that for $j=1,\dots,k$ we have defined some partitions $\{A_n^j\}_{n=1}^\infty$ of $A$ and some collections of simple functions $\{i_n^j:A_n^j\to\sR\}_{n=1}^\infty$, $\{s_n^j:A_n^j\to\sR\}_{n=1}^\infty$ such that $ i^j_n\leq f\leq s^j_n$ for each $n$, $\{A_n^{j+1}\}_{n=1}^\infty$ is a refinement of $\{A_n^j\}_{n=1}^\infty$ for each $j$, the series $ \sum\limits_{n=1}^\infty \int_{A_n}|s^j_n|\,dx$ and $ \sum\limits_{n=1}^\infty \int_{A_n}|i_n^j|\,dx$ both converge in $\sR$ for each $j$ and
    \[
    \sum\limits_{n=1}^\infty \int_{A_n^j}(s_n^j-i_n^j)\,dx<d^j.
    \]
    Now, by definition, there exists a partition $\{A_n\}_{n=1}^\infty$ of $A$ and two collections of simple functions ${\{i_n:A_n\to\sR\}_{n=1}^\infty}$ and $\{s_n:A_n\to\sR\}_{n=1}^\infty$ such that $ i_n\leq f\leq s_n$, the series $ \sum\limits_{n=1}^\infty \int_{A_n}|s_n|\,dx$ and $ \sum\limits_{n=1}^\infty \int_{A_n}|i_n|\,dx$ both converge in $\sR$ and
    \[
    \sum\limits_{n=1}^\infty \int_{A_n}(s_n-i_n)\,dx<d^{k+1}.
    \]
    We let $\{A_t^{k+1}\}_{t=1}^\infty=\{R_{n,m}\}_{n,m=1}^\infty$ be the common refinement of $\{A_n\}_{n=1}^\infty$ and $\{A_m^k\}_{m=1}^\infty$ and we let
    \[
    i_t^{k+1}=i_{n,m}^{k+1}:=\max\{i_n,i_m^k\}\text{ and }
    s_t^{k+1}=s_{n,m}^{k+1}:=\min\{s_n,s_m^k\}.
    \]
    By \cref{simple-measurable}, the sets $I_{n,m}:=\{x\in R_{n,m}:i_n\leq i_m^k\}$ and $S_{n,m}:=\{x\in R_{n,m}:s_n\geq s_m^k\}$ are measurable, and thus, for $t\in\N$
    \begin{align*}
        \int_{A_t}|s_t^{k+1}|\,dx&=\int_{R_{n,m}}|s_{n,m}^{k+1}|\,dx\\
        &=\int_{R_{n,m}\cap S_{n,m}}|s_{n,m}^{k+1}|\,dx+\int_{R_{n,m}\bkl S_{n,m}}|s_{n,m}^{k+1}|\,dx\\
        &=\int_{R_{n,m}\cap S_{n,m}}|s_{m}^{k}|\,dx+\int_{R_{n,m}\bkl S_{n,m}}|s_n|\,dx\\
        &\leq \int_{R_{n,m}}|s_{m}^{k}|\,dx+\int_{R_{n,m}}|s_n|\,dx.
    \end{align*}
    Similarly
    \[
     \int_{A_t}|i_t^{k+1}|\,dx\leq \int_{R_{n,m}}|i_{m}^{k}|\,dx+\int_{R_{n,m}}|i_n|\,dx.
    \]
    It follows that the series $ \sum\limits_{t=1}^\infty \int_{A_t}|s_t^{k+1}|\,dx$ and $ \sum\limits_{t=1}^\infty \int_{A_t}|i_t^{k+1}|\,dx$ both converge in $\sR$ and
    \begin{align*}
         \sum\limits_{t=1}^\infty \int_{A^{k+1}_t}(s^{k+1}_t-i^{k+1}_t)\,dx&=\sum\limits_{n=1}^\infty\sum\limits_{m=1}^\infty\int_{A_n\cap A_m^k}(s^{k+1}_{n,m}-i^{k+1}_{n,m})\,dx\\
         &\leq \sum\limits_{n=1}^\infty\sum\limits_{m=1}^\infty\int_{A_n\cap A_m^k}(s_n-i_n)\,dx\\
         &=\sum\limits_{n=1}^\infty \int_{A_n}(s_n-i_n)\,dx\\
         &<d^{k+1}.
    \end{align*}
    Thus, the sequence $(\{A_n^k\}_{n=1}^\infty)_{k=1}^\infty$ of partitions of $A$ and the two sequences of collections of simple functions $(\{i_n^k:A_n^k\to\sR\}_{n=1}^\infty)_{k=1}^\infty$, $(\{s_n^k:A_n^k\to\sR\}_{n=1}^\infty)_{k=1}^\infty$ satisfy that
    \begin{align*}
        0\leq\sum\limits_{m=1}^\infty \int_{A_m^{k+1}}i^{k+1}_m\,dx-\sum\limits_{n=1}^\infty \int_{A_n^k}i^k_n\,dx&=\sum\limits_{m=1}^\infty \sum\limits_{n=1}^\infty  \int_{A_n^k\cap A_m^{k+1}}i^{k+1}_m\,dx-\sum\limits_{n=1}^\infty \int_{A_n^k}i^k_n\,dx\\
        &\leq \sum\limits_{m=1}^\infty \sum\limits_{n=1}^\infty  \int_{A_n^k\cap A_m^{k+1}}s^{k}_n\,dx-\sum\limits_{n=1}^\infty \int_{A_n^k}i^k_n\,dx\\
        &=\sum\limits_{n=1}^\infty \sum\limits_{m=1}^\infty  \int_{A_n^k\cap A_m^{k+1}}s^{k}_n\,dx-\sum\limits_{n=1}^\infty \int_{A_n^k}i^k_n\,dx\\
        &=\sum\limits_{n=1}^\infty \int_{A_n^k}s^k_n\,dx-\sum\limits_{n=1}^\infty \int_{A_n^k}i^k_n\,dx\\
        &<d^k
    \end{align*}
    and
    \begin{align*}
        0\leq\sum\limits_{n=1}^\infty \int_{A_n^k}s^k_n\,dx-\sum\limits_{m=1}^\infty \int_{A_m^{k+1}}s^{k+1}_m\,dx&=\sum\limits_{n=1}^\infty \int_{A_n^k}s^k_n\,dx-\sum\limits_{m=1}^\infty \sum\limits_{n=1}^\infty  \int_{A_n^k\cap A_m^{k+1}}s^{k+1}_m\,dx\\
        &\leq \sum\limits_{n=1}^\infty \int_{A_n^k}s^k_n\,dx-\sum\limits_{m=1}^\infty \sum\limits_{n=1}^\infty  \int_{A_n^k\cap A_m^{k+1}}i^{k}_n\,dx\\
        &=\sum\limits_{n=1}^\infty \int_{A_n^k}s^k_n\,dx-\sum\limits_{n=1}^\infty \sum\limits_{m=1}^\infty  \int_{A_n^k\cap A_m^{k+1}}i^{k}_n\,dx\\
        &=\sum\limits_{n=1}^\infty \int_{A_n^k}s^k_n\,dx-\sum\limits_{n=1}^\infty \int_{A_n^k}i^k_n\,dx\\
        &<d^k.
    \end{align*}
    Hence, the sequences $\left(\sum\limits_{n=1}^\infty \int_{A_n^k}i^k_n\,dx\right)_{k=1}^\infty$ and $\left(\sum\limits_{n=1}^\infty \int_{A_n^k}s^k_n\,dx\right)_{k=1}^\infty$ are Cauchy and therefore convergent. It is also easy to check that their limits are equal.

    We define
    \[
    I_{f,A}:=\lim\limits_{k\to\infty}\sum\limits_{n=1}^\infty \int_{A_n^k}i^k_n\,dx=\lim\limits_{k\to\infty}\sum\limits_{n=1}^\infty \int_{A_n^k}s^k_n\,dx.
    \]
    Now we show that
    \[
    I_{f,A}=\inf(U_{f,A})=\sup(I_{f,A}).
    \]
    Let $\{A_n\}_{n=1}^\infty$ be a partition of $A$ and let $\{i_n:A_n\to\sR\}_{n=1}^\infty$ be a collection of simple functions such that $ i_n\leq f$ and the series $ \sum\limits_{n=1}^\infty \int_{A_n}|i_n|\,dx$ converges in $\sR$. Now, given $k\in\N$, we have that
    \begin{align*}
        \sum\limits_{n=1}^\infty \int_{A_n}i_n\,dx&=\sum\limits_{n=1}^\infty \sum\limits_{m=1}^\infty  \int_{A_m^k\cap A_n}i_n\,dx\\
        &\leq \sum\limits_{n=1}^\infty \sum\limits_{m=1}^\infty  \int_{A_m^k\cap A_n}s_m^k\,dx\\
        &=\sum\limits_{m=1}^\infty \sum\limits_{n=1}^\infty  \int_{A_m^k\cap A_n}s_m^k\,dx\\
        &=\sum\limits_{m=1}^\infty \int_{A_m^k}s_m^k\,dx.
    \end{align*}
    Taking the limit as $k\to\infty$ yields
    \[
    \sum\limits_{n=1}^\infty \int_{A_n}i_n\,dx\leq I_{f,A}.
    \]
    Thus,
    \[
    \sup(L_{f,A})=I_{f,A}.
    \]
    Similarly one checks that
    \[
    \inf(U_{f,A})=I_{f,A},
    \]
    which proves the result.
\end{proof}
\begin{theorem}[Characterization of measurable functions]\label{char}
    Let $A\subseteq \sR$ and $f:A\to\sR$ be measurable. Then, there exist a sequence of partitions $(\{A_n^k\}_{n=1}^\infty)_{k=1}^\infty$ of $A$ and two sequences of collections of simple functions $(\{i_n^k:A_n^k\to\sR\}_{n=1}^\infty)_{k=1}^\infty$, $(\{s_n^k:A_n^k\to\sR\}_{n=1}^\infty)_{k=1}^\infty$ such that $i^k_n\leq f\leq s^k_n$, the series $ \sum\limits_{n=1}^\infty \int_{A_n}|s_n^k|\,dx$ and $ \sum\limits_{n=1}^\infty \int_{A_n}|i^k_n|\,dx$ both converge in $\sR$ and for all $\epsilon,\delta>0$ in $\sR$, there exists some measurable set $U\subseteq A$ such that $m(U)<\delta$ and the sequences of functions $(L_k:A\to\sR)_{k=1}^\infty$ and $(S_k:A\to\sR)_{k=1}^\infty$ given by
    \[
    L_k:=\sum_{n=1}^\infty i_n^k\chi_{A_n^k} \text{ and }S_k:=\sum_{n=1}^\infty s_n^k\chi_{A_n^k}
    \]
    satisfy that there exists some $K\in \N$ such that if $k\geq K$, then, for all $x\in A\bkl U$
    \[
    S_k(x)-L_k(x)<\epsilon.
    \]
\end{theorem}
\begin{proof}
    Let us suppose that $f$ is measurable. We take the partitions $(\{A_n^k\}_{n=1}^\infty)_{k=1}^\infty$ of $A$ and the two sequences of collections of simple functions $(\{i_n^k:A_n^k\to\sR\}_{n=1}^\infty)_{k=1}^\infty$, $(\{s_n^k:A_n^k\to\sR\}_{n=1}^\infty)_{k=1}^\infty$ to be the ones built in the proof of \cref{R-Criteria}. We  define
    \[
    U_k^m:=\left\{x\in A\mid S_k(x)-L_k(x)>d^m\right\}=\bigcup\limits_{n=1}^\infty\left\{x\in A_n^k\mid s_n^k(x)-i_n^k(x)>d^m\right\}.
    \]
    Since $m(A_n^k)\to 0$ as $k\to\infty$, it follows, from \cref{simple-measurable}, that the sets $U_k^m$ are measurable. Also
    \begin{align*}
        d^mm(U_k^m)&=\int_{U_k^m}d^m\,dx\\
        &=\sum\limits_{n=1}^\infty\int_{U_k^m\cap A_n^k}d^m\,dx\\
        &\leq\sum\limits_{n=1}^\infty\int_{U_k^m\cap A_n^k}(s_n^k-i_n^k)\,dx\\
        &\leq \sum\limits_{n=1}^\infty\int_{A_n^k}(s_n^k-i_n^k)\,dx\\
        &<d^k.
    \end{align*}
    Taking the limit as $k\to\infty$ yields that the set $U_m:=\bigcap\limits_{k=1}^\infty U_m^k$ is measurable and has measure $m(U_m)=\lim\limits_{k\to\infty}m(U_m^k)=0$. Now let $\epsilon,\delta>0$ in $\sR$ be given. Let $m\in\N$ be such that $d^m<\epsilon$ and let $K\in\N$ be such that $m(U_m^K)<\delta$. We claim that $U:=U_m^K$ is the desired set.

    Let $x\in A\bkl U$ and let $k\geq K$. We note that $L_k\leq L_{k+1}\leq S_{k+1}\leq S_k$. Therefore, $U_m^k\subseteq U_m^K$ and $A\bkl U_m^K\subseteq A\bkl U_m^k$. Thus, for $x\in A\bkl U$, $S_k(x)-L_k(x)<\epsilon$.\\
\end{proof}
\begin{propn}
    The converse of \cref{char} holds true whenever $f$ is bounded.
\end{propn}
    We leave the proof of this result as an exercise to the reader.\\

    \textit{Hint: $\int_A \left(S_k-L_k\right)\,dx=\int_U \left(S_k-L_k\right)\,dx+\int_{A\bkl U}\left(S_k-L_k\right)\,dx$}.
\section{Integration on $\sR$}\label{S5}
In light of \cref{R-Criteria} we make the following definition:
\begin{defn}
    Let $A\subseteq \sR$ and $f:A\to\sR$ be measurable. We define the integral of $f$ over $A$ to be
    \[
    \int_A f\,dx:=\inf(U_{f,A})=\sup(L_{f,A}).
    \]
\end{defn}
\subsection{Basic properties}
\begin{propn}\label{aux1}
    Let $A\subseteq \sR$ be measurable and let $f:A\to\sR$ be simple. Then $f$ is measurable and the integral of $f$ as a measurable function coincides with the integral of $f$ as a simple function. In particular, if $\alpha\in\sR$
    \[
    \int_A\alpha\,dx=\alpha m(A).
    \]
\end{propn}
\begin{proof}
    The proof follows immediately from the definition of the integral of a measurable function and from the properties of the integral of a simple funtion.
\end{proof}
\begin{lemma}\label{tool}
    Let $A\subseteq \sR$ and $f:A\to\sR$ be measurable. Then for any sequence of partitions $(\{A_n^k\}_{n=1}^\infty)_{k=1}^\infty$ of $A$ and any two sequences of collections of simple functions $(\{i_n^k:A_n^k\to\sR\}_{n=1}^\infty)_{k=1}^\infty$, $(\{s_n^k:A_n^k\to\sR\}_{n=1}^\infty)_{k=1}^\infty$ such that $ i^k_n\leq f\leq s^k_n$, the series $ \sum\limits_{n=1}^\infty \int_{A_n}|s_n^k|\,dx$ and $ \sum\limits_{n=1}^\infty \int_{A_n}|i^k_n|\,dx$ both converge in $\sR$ and
    \[
     \lim\limits_{k\to\infty}\sum\limits_{n=1}^\infty \int_{A_n^k}(s_n^k-i^k_n)\,dx=0,
    \]
    we have that both $\lim\limits_{k\to\infty}\sum\limits_{n=1}^\infty \int_{A_n^k}s^k_n$ and $\lim\limits_{k\to\infty}\sum\limits_{n=1}^\infty \int_{A_n^k}i^k_n$ exist and
    \[
    \int_A f\,dx=\lim\limits_{k\to\infty}\sum\limits_{n=1}^\infty \int_{A_n^k}s^k_n=\lim\limits_{k\to\infty}\sum\limits_{n=1}^\infty \int_{A_n^k}i^k_n.
    \]
\end{lemma}
\begin{proof}
    We simply observe that
    \[
    0\leq \sum\limits_{n=1}^\infty \int_{A_n^k}s^k_n-\int_A f\,dx\leq \sum\limits_{n=1}^\infty \int_{A_n^k}(s_n^k-i^k_n)\,dx.
    \]
    Similarly
    \[
    0\leq \int_A f\,dx-\sum\limits_{n=1}^\infty \int_{A_n^k}i^k_n\leq \sum\limits_{n=1}^\infty \int_{A_n^k}(s_n^k-i^k_n)\,dx.
    \]
\end{proof}
\begin{propn}
    Let $A\subseteq \sR$ and $f:A\to\sR$ be S-measurable. Then, $f$ is measurable. Moreover, if $s\int_A f\,dx$ denotes the S-integral of $f$, then
    \[
    s\int_A f\,dx=\int_A f\,dx.
    \]
\end{propn}
\begin{proof}
    Let $\epsilon>0$ in $\sR$ be given and let $M>0$ in $\sR$ be a bound for $f$. Since $f$ is S-measurable, there exists some collection of closed intervals $\{J_n\}_{n=1}^\infty$ whose interiors do not overlap such that for each $n\in\N$, $J_n\subseteq A$, $f$ is $\sR$-analytic on each $J_n$ and
    \[
    m(A)-\sum\limits_{n=1}^\infty l(J_n)<\dfrac{d\epsilon}{M}.
    \]
    We note that $\{J_n\}_{n=1}^\infty\cup \left\{A\bkl\bigcup\limits_{n=1}^\infty J_n\right\}$ is a partition of $A$ and that $\{i_n:A_n\to\sR\}_{n=1}^\infty\cup \left\{i_0:A\bkl\bigcup\limits_{n=1}^\infty J_n\to\sR\right\}$, $\{s_n:A_n\to\sR\}_{n=1}^\infty\cup \left\{s_0:A\bkl\bigcup\limits_{n=1}^\infty J_n\to\sR\right\}$ given by
    \begin{align*}
        s_0&:=M\\
        i_0&:=-M\\
        s_n&:=f \text{ for $n\geq 1$}\\
        i_n&:=f \text{ for $n\geq 1$}
    \end{align*}
    are two collection of simple functions such that $i_n\leq f\leq s_n$, the series $ \sum\limits_{n=0}^\infty \int_{A_n}|s_n|\,dx$ and $ \sum\limits_{n=0}^\infty \int_{A_n}|i_n|\,dx$ both converge in $\sR$ and
    \[
    \sum\limits_{n=0}^\infty \int_{A_n}(s_n-i_n)\,dx=\int_{A\bkl\bigcup\limits_{n=1}^\infty J_n}2M\,dx=2Mm\left(A\bkl\bigcup\limits_{n=1}^\infty J_n\right)<\epsilon.
    \]
    Thus, $f$ is measurable. To show that
    \[
    s\int_A f\,dx=\int_A f\,dx,
    \]
    one simply checks that
    \[
    \int_A f\,dx-s\int_A f\,dx=\lim\limits_{\epsilon\to 0}\sum\limits_{n=0}^\infty \int_{A_n}s_n\,dx-\lim\limits_{\epsilon\to 0}\sum\limits_{n=1}^\infty \int_{A_n}s_n\,dx=M\lim\limits_{\epsilon\to 0}m\left(A\bkl\bigcup\limits_{n=1}^\infty J_n\right)=0.
    \]
\end{proof}
\begin{propn}[Linearity]
    Let $A\subseteq \sR$ and $f,g:A\to\sR$ be measurable and let $\alpha \in\sR$. Then
    \[
    \int_A (\alpha f+g)\,dx=\alpha\int_Af\,dx+\int_A g\,dx.
    \]
\end{propn}
\begin{proof}
    Without loss of generality let $\alpha>0$. Let $(\{A_n^k\}_{n=1}^\infty)_{k=1}^\infty$ be a sequence of partitions of $A$ and $(\{i_n^k:A_n^k\to\sR\}_{n=1}^\infty)_{k=1}^\infty$, ${(\{s_n^k:A_n^k\to\sR\}_{n=1}^\infty)_{k=1}^\infty}$, $(\{l_n^k:A_n^k\to\sR\}_{n=1}^\infty)_{k=1}^\infty$, ${(\{u_n^k:A_n^k\to\sR\}_{n=1}^\infty)_{k=1}^\infty}$ be a sequence of collections of simple functions such that $ i^k_n\leq f\leq s^k_n$, $ l^k_n\leq g\leq u^k_n$, all the series $\sum\limits_{n=1}^\infty \int_{A_n^k}|s_n^k|\,dx$, $\sum\limits_{n=1}^\infty \int_{A_n^k}|i_n^k|\,dx$, $\sum\limits_{n=1}^\infty \int_{A_n^k}|u_n^k|\,dx$ and $\sum\limits_{n=1}^\infty \int_{A_n^k}|l_n^k|\,dx$ converge in $\sR$ and
    \[
    \lim\limits_{k\to\infty}\sum\limits_{n=1}^\infty \int_{A_n^k}(s_n^k-i^k_n)\,dx=\lim\limits_{k\to\infty}\sum\limits_{n=1}^\infty \int_{A_n^k}(u_n^k-l^k_n)\,dx=0.
    \]
    Then, $\alpha  i^k_n+l_n^k\leq\alpha f+g\leq\alpha s^k_n+u_n^k$ and
    \[
    \lim\limits_{k\to\infty}\sum\limits_{n=1}^\infty \int_{A_n^k}(\alpha s^k_n+u_n^k-(\alpha  i^k_n+l_n^k))\,dx=\alpha\lim\limits_{k\to\infty}\sum\limits_{n=1}^\infty \int_{A_n^k}(s_n^k-i^k_n)\,dx-\lim\limits_{k\to\infty}\sum\limits_{n=1}^\infty \int_{A_n^k}(u_n^k-l^k_n)\,dx=0.
    \]
    Thus,
    \[
    \int_A \alpha f+g\,dx=\lim\limits_{k\to\infty}\sum\limits_{n=1}^\infty \int_{A_n^k}\alpha s^k_n+u_n^k\,dx=\alpha\lim\limits_{k\to\infty}\sum\limits_{n=1}^\infty \int_{A_n^k}s^k_n\,dx+\lim\limits_{k\to\infty}\sum\limits_{n=1}^\infty \int_{A_n^k}u_n^k\,dx=\alpha\int_A f\,dx+\int_A g\,dx.
    \]
\end{proof}
\begin{propn}\label{aux2}
    Let $A\subseteq \sR$ and $f:A\to\sR$ be measurable with $f\geq 0$. Then
    \[
    \int_A f\,dx\geq 0.
    \]
\end{propn}
\begin{proof}
    Since $f$ is measurable, then, given any partition $\{A_n\}_{n=1}^\infty$ of $A$ and any sequence of simple functions $\{s_n:A_n\to\sR\}$ such that $s_n\geq f\geq 0$ we have that
    \[
    \sum\limits_{n=1}^\infty\int_{A_n}s_n\geq 0
    \]
    and hence
    \[
    \int_Af\,dx=\inf(U_{f,A})\geq0.
    \]
\end{proof}
\begin{col}\label{aux3}
    Let $A\subseteq \sR$ and $f,g:A\to\sR$ be measurable such that $f\leq g$. Then
    \[
    \int_Af\,dx\leq \int_Ag\,dx.
    \]
\end{col}
\begin{col}\label{aux4}
    Let $A\subseteq \sR$ and $f:A\to\sR$ be measurable and let $|f|\leq M$ for some $M\in\sR$. Then
    \[
    \left|\int_Af\,dx\right|\leq Mm(A).
    \]
\end{col}
\begin{propn}[Additivity]
    Let $A\subseteq \sR$ and $f:A\to\sR$ be measurable, and let $B\subseteq A$ be measurable. Then
    \[
    \int_Af\,dx=\int_Bf\,dx+\int_{A\bkl B}f\,dx.
    \]
\end{propn}
\begin{proof}
    By \cref{rest-is-measurable}, $f$ is measurable on $B$ and $A\bkl B$. Now, let $(\{A_n^k\}_{n=1}^\infty)_{k=1}^\infty$ be a sequence of partitions of $A$ and $(\{i_n^k:A_n^k\to\sR\}_{n=1}^\infty)_{k=1}^\infty$, ${(\{s_n^k:A_n^k\to\sR\}_{n=1}^\infty)_{k=1}^\infty}$ be sequences of collections of simple functions such that $i^k_n\leq f\leq s^k_n$, the series $\sum\limits_{n=1}^\infty \int_{A_n^k}|s_n^k|\,dx$ and $\sum\limits_{n=1}^\infty \int_{A_n^k}|i_n^k|\,dx$ both converge in $\sR$ and
    \[
    \lim\limits_{k\to\infty}\sum\limits_{n=1}^\infty \int_{A_n^k}(s_n^k-i^k_n)\,dx=0.
    \]
    We notice that $(\{B\cap A_n^k\}_{n=1}^\infty)_{k=1}^\infty$ and $(\{A_n^k\bkl B\}_{n=1}^\infty)_{k=1}^\infty$ are partitions of $B$ and $A\bkl B$ respectively. Moreover, by \cref{tool}
    \begin{align*}
        \int_Af\,dx&=\lim\limits_{k\to\infty}\sum\limits_{n=1}^\infty \int_{A_n^k}s_n^k\,dx\\
        &=\lim\limits_{k\to\infty}\left(\sum\limits_{n=1}^\infty \int_{B\cap A_n^k}s_n^k\,dx+\sum\limits_{n=1}^\infty \int_{ A_n^k\bkl B}s_n^k\,dx\right)\\
        &=\lim\limits_{k\to\infty}\sum\limits_{n=1}^\infty \int_{B\cap A_n^k}s_n^k\,dx+\lim\limits_{k\to\infty}\sum\limits_{n=1}^\infty \int_{A_n^k\bkl B}s_n^k\,dx\\
        &=\int_{B}f\,dx+\int_{A\bkl B}f\,dx.
    \end{align*}
\end{proof}
\begin{propn}\label{bounded-abs-is-measurable}
    Let $A\subseteq \sR$ be measurable and let $f:A\to\sR$ be measurable and bounded. Then, $|f|$ is measurable.
\end{propn}
\begin{proof}
    Let $\epsilon>0$ in $\sR$ be given and let $M>0$ in $\sR$ be a bound for $f$. By \cref{char}, there exists a measurable set $U\subseteq A$, a partition $\{A_n\}_{n=1}^\infty$ of $A$, a collection of simple functions $\{i_n:A_n\to\sR\}_{n=1}^\infty$, $\{s_n:A_n\to\sR\}_{n=1}^\infty$ such that $i_n\leq f\leq s_n$, the series $ \sum\limits_{n=1}^\infty \int_{A_n}|s_n|\,dx$ and $ \sum\limits_{n=1}^\infty \int_{A_n}|i_n|\,dx$ both converge in $\sR$, $m(U)<\dfrac{d\epsilon}{M}$,
    \[
    \sum\limits_{n=1}^\infty \int_{A_n}(s_n-i_n)\,dx< d\epsilon
    \]
    and, for all $x\in B:=A\bkl U$
    \[
    S(x)-L(x)<\epsilon_0,
    \]
    where
    \[
    \epsilon_0:=\dfrac{d\epsilon}{m(A)}
    \]
    and
    \[
    L:=\sum_{n=1}^\infty i_n\chi_{A_n} \text{ and }S:=\sum_{n=1}^\infty s_n\chi_{A_n}.
    \]
    We define $B_-:=B\cap S^{-1}(-\infty,0)$, $B_0:=B\cap S^{-1}[0,\epsilon_0]$ and $B_+:=B\cap S^{-1}(\epsilon_0,\infty)$. Using the same argument as in \cref{char}, the sets $B_-$, $B_0$ and $B_+$ are measurable. Moreover, it's easy to check that
    \begin{align*}
        x\in B_-&\implies f(x),L(x),S(x)< 0\\
        x\in B_+&\implies f(x),L(x),S(x)>0\\
        x\in B_0&\implies |f(x)|\leq \epsilon_0.
    \end{align*}
    We define $B_n:=A_n\bkl U$, $U_n:=A_n\cap U$, $B_n^+:=B_n\cap B_+$, $B_n^-:=B_n\cap B_-$ and $B_n^0:=B_n\cap B_0$. Now, $\{U_n\}_{n=1}^\infty\cup\{B_n^+\}_{n=1}^\infty\cup\{B_n^-\}_{n=1}^\infty\cup\{B_n^0\}_{n=1}^\infty$ is a partition of $A$ and we define the family of simple functions $\{i_n^U:U_n\to\sR\}_{n=1}^\infty$, $\{s_n^U:U_n\to\sR\}_{n=1}^\infty$, $\{i_n^+:B_n^+\to\sR\}_{n=1}^\infty$, $\{s_n^+:B_n^+\to\sR\}_{n=1}^\infty$, $\{i_n^-:B_n^-\to\sR\}_{n=1}^\infty$, $\{s_n^-:B_n^-\to\sR\}_{n=1}^\infty$ and $\{i_n^0:B_n^0\to\sR\}_{n=1}^\infty$, $\{s_n^0:B_n^0\to\sR\}_{n=1}^\infty$ by
    \begin{align*}
        i_n^U(x)&:=0\\
        s_n^U(x)&:=M\\
        i_n^+(x)&:=i_n(x)\\
        s_n^+(x)&:=s_n(x)\\
        i_n^-(x)&:=-s_n(x)\\
        s_n^-(x)&:=-i_n(x)\\
        i_n^0(x)&:=0\\
        s_n^0(x)&:=\epsilon_0.
    \end{align*}
    It's easy to check that for $*\in\{U,+,-,0\}$, $i_n^*$, $s_n^*$ are simple, $i_n^*\leq |f|\leq s_n^*$ and that the series
    \[
    \sum\limits_{n=1}^\infty \left(\int_{U_n}|s_n^U|\,dx+\int_{B_n^+}|s_n^+|\,dx+\int_{B_n^-}|s_n^-|\,dx+\int_{B_n^0}|s_n^0|\,dx\right)
    \]
    and
    \[
    \sum\limits_{n=1}^\infty \left(\int_{U_n}|i_n^U|\,dx+\int_{B_n^+}|i_n^+|\,dx+\int_{B_n^-}|i_n^-|\,dx+\int_{B_n^0}|i_n^0|\,dx\right)
    \]
    both converge in $\sR$. Moreover,
    \begin{align*}
        \sum\limits_{n=1}^\infty &\left(\int_{U_n}(s_n^U-i_n^U)\,dx+\int_{B_n^+}(s_n^+-i_n^+)\,dx+\int_{B_n^-}(s_n^--i^-_n)\,dx+\int_{B_n^0}(s_n^0-i_n^0)\,dx\right)\\
        &=\sum\limits_{n=1}^\infty \left(\int_{U_n}M\,dx+\int_{B_n^+}(s_n-i_n)\,dx+\int_{B_n^-}(s_n-i_n)\,dx+\int_{B_n^0}\epsilon_0\,dx\right)\\
        &=Mm(U)+\sum\limits_{n=1}^\infty \left(\int_{B_n^+}(s_n-i_n)\,dx+\int_{B_n^-}(s_n-i_n)\,dx+\int_{B_n^0}\epsilon_0\,dx\right)\\
        &\leq Mm(U)+\sum\limits_{n=1}^\infty \left(\int_{B_n^+}(s_n-i_n)\,dx+\int_{B_n^-}(s_n-i_n)\,dx+\int_{A_n}\epsilon_0\,dx\right)\\
        &=Mm(U)+\epsilon_0m(A)+\sum\limits_{n=1}^\infty \left(\int_{B_n^+}(s_n-i_n)\,dx+\int_{B_n^-}(s_n-i_n)\,dx\right)\\
        &\leq Mm(U)+\epsilon_0m(A)+\sum\limits_{n=1}^\infty\int_{A_n} (s_n-i_n)\,dx\\
        &<Mm(U)+\epsilon_0m(A)+d\epsilon\\
        &<3d\epsilon\\
        &<\epsilon.
    \end{align*}
    Hence $|f|$ is measurable.
\end{proof}
\begin{col}
    Let $A\subseteq \sR$ be measurable and let $f,g:A\to\sR$ be measurable and bounded. Then, the functions $\min\{f,g\}$ and $\max\{f,g\}$ are measurable.
\end{col}
\begin{col}
    Let $A\subseteq \sR$ be measurable and let $f:A\to\sR$ be bounded. Then, $f$ is measurable if and only if $f_+:=\max\{f,0\}$ and $f_-:=\max\{-f,0\}$ are measurable.
\end{col}\begin{propn}
    Let $A\subseteq \sR$ be measurable and let $f,g:A\to\sR$ be measurable and bounded. Then $f\cdot g$ is measurable.
\end{propn}
\begin{proof} Without loss of generality, let $f,g\geq0$\\

    Let $M>0$ in $\sR$ be a bound for $f$ and $g$ on $A$ and let $\epsilon>0$ in $\sR$ be given. Then, let $\{A_n\}_{n=1}^\infty$ be a partition of $A$ and let $\{i_n:A_n\to\sR\}_{n=1}^\infty$, $\{s_n:A_n\to\sR\}_{n=1}^\infty$, $\{l_n:A_n\to\sR\}_{n=1}^\infty$, $\{u_n:A_n\to\sR\}_{n=1}^\infty$ be collections of simple functions such that $ 0\leq i_n\leq f\leq s_n$, $0\leq l_n\leq g\leq u_n\leq M$, the series $ \sum\limits_{n=1}^\infty \int_{A_n}s_n\,dx$, $ \sum\limits_{n=1}^\infty \int_{A_n}u_n\,dx$, $ \sum\limits_{n=1}^\infty \int_{A_n}l_n\,dx$ and $ \sum\limits_{n=1}^\infty \int_{A_n}i_n\,dx$ all converge in $\sR$ and
    \[
    \sum\limits_{n=1}^\infty \int_{A_n}(s_n-i_n)\,dx+\sum\limits_{n=1}^\infty \int_{A_n}(u_n-l_n)\,dx<\dfrac{d\epsilon}{M}.
    \]
    We note that $0\leq i_n\cdot l_n\leq f\cdot g\leq s_n\cdot u_n$ and that
    \[
    s_n\cdot u_n-i_n\cdot l_n=u_n(s_n-i_n)+i_n(u_n-l_n)\leq M(s_n-i_n)+M(u_n-l_n).
    \]
    Thus,
    \[
    \sum\limits_{n=1}^\infty \int_{A_n}(s_n\cdot u_n-i_n\cdot l_n)\,dx\leq\sum\limits_{n=1}^\infty \int_{A_n}M(s_n-i_n)\,dx+\sum\limits_{n=1}^\infty \int_{A_n}M(u_n-l_n)\,dx<\epsilon.
    \]
\end{proof}
\begin{propn}[Countable additivity]\label{w-additivity}
    Let $A\subseteq\sR$ and $f:A\to\sR$ be measurable and let $\{A_m\}_{m=1}^\infty$ be a partition of $A$. Then
    \[
    \int_Af\,dx=\sum\limits_{m=1}^\infty \int_{A_m}f\,dx.
    \]
\end{propn}
\begin{proof}
    By \cref{rest-is-measurable}, $f$ is measurable on each $A_m$. Now, let $(\{A_n^k\}_{n=1}^\infty)_{k=1}^\infty$ be a sequence of partitions of $A$ and $(\{i_n^k:A_n^k\to\sR\}_{n=1}^\infty)_{k=1}^\infty$, ${(\{s_n^k:A_n^k\to\sR\}_{n=1}^\infty)_{k=1}^\infty}$ be a sequence of collections of simple functions such that $ i^k_n\leq f\leq s^k_n$, the series $\sum\limits_{n=1}^\infty \int_{A_n^k}|s_n^k|\,dx$ and $\sum\limits_{n=1}^\infty \int_{A_n^k}|i_n^k|\,dx$ both converge in $\sR$ and
    \[
    \lim\limits_{k\to\infty}\sum\limits_{n=1}^\infty \int_{A_n^k}(s_n^k-i^k_n)\,dx=0.
    \]
    We note that for each $m\in\N$, $\{A_m\cap A_n^k\}_{n=1}^\infty$ is a partition of $A_m$ and that
    \[
    -\sum\limits_{n=1}^\infty \int_{A_n^k\cap A_m}|i_n^k|\,dx\leq \sum\limits_{n=1}^\infty \int_{A_n^k\cap A_m}i_n^k\,dx\leq \int_{A_m}f\,dx\leq\sum\limits_{n=1}^\infty \int_{A_n^k\cap A_m}s_n^k\,dx\leq \sum\limits_{n=1}^\infty \int_{A_n^k\cap A_m}|s_n^k|\,dx.
    \]
    Thus, the series $\sum\limits_{m=1}^\infty \int_{A_m}f\,dx$ converges in $\sR$. Furthermore,
    \begin{align*}
         \int_Af\,dx-\sum\limits_{m=1}^\infty \int_{A_m}f\,dx&\leq \sum\limits_{n=1}^\infty \int_{A_n^k}s_n^k-\sum\limits_{m=1}^\infty\sum\limits_{n=1}^\infty \int_{A_n^k\cap A_m}i_n^k\,dx\\
         &=\sum\limits_{n=1}^\infty \int_{A_n^k}s_n^k-\sum\limits_{n=1}^\infty\sum\limits_{m=1}^\infty \int_{A_n^k\cap A_m}i_n^k\,dx\\
         &=\sum\limits_{n=1}^\infty \int_{A_n^k}(s_n^k-i^k_n)\,dx.
    \end{align*}
    Similarly
    \[
    \sum\limits_{m=1}^\infty \int_{A_m}f\,dx-\int_Af\,dx\leq \sum\limits_{n=1}^\infty \int_{A_n^k}(s_n^k-i^k_n)\,dx.
    \]
    Hence
    \[
    \left|\sum\limits_{m=1}^\infty \int_{A_m}f\,dx-\int_Af\,dx\right|\leq \sum\limits_{n=1}^\infty \int_{A_n^k}(s_n^k-i^k_n)\,dx\underset{k\to\infty}{\to}0,
    \]
    which proves the result.
\end{proof}
\begin{propn}\label{partition-positive}
    Let $A\subseteq\sR$ be measurable, let $f:A\to\sR$ be a non-negative function and let $\{A_m\}_{m=1}^\infty$ be a partition of $A$ such that $f$ is measurable and bounded on each $A_m$ and the series
    \[
    \sum\limits_{m=1}^\infty \int_{A_m}f\,dx
    \]
    converges. Then, $f$ is measurable on $A$.
\end{propn}
\begin{proof}
    Let $\epsilon>0$ in $\sR$ be given. For each $m\in\N$, we can find partitions $\{A_n^m\}_{n=1}^\infty$ of $A_m$ and collections of simple functions $\{i_n^m:A_n^m\to\sR\}_{n=1}^\infty$, $\{s_n^m:A_n^m\to\sR\}_{n=1}^\infty$ such that $0\leq i_n^m\leq f\leq s_n^m$, the series $\sum\limits_{n=1}^\infty \int_{A_n^m}s_n^m\,dx$ and $\sum\limits_{n=1}^\infty \int_{A_n^m}i_n^m\,dx$ both converge in $\sR$ and
    \[
    \sum\limits_{n=1}^\infty \int_{A_n^m}(s_n^m-i_n^m)\,dx<d^m\epsilon.
    \]
    We notice that $\{A_n^m\}_{n,m=1}^\infty$ is a partition of $A$. Furthermore,
    \begin{align*}
          \sum\limits_{n=1}^\infty \int_{A_n^m}|s_n^m|\,dx-\int_{A_m}f\,dx&=\sum\limits_{n=1}^\infty \int_{A_n^m}s_n^m\,dx-\int_{A_m}f\,dx\\
          &=\sum\limits_{n=1}^\infty \int_{A_n^m}s_n^m\,dx-\sum\limits_{n=1}^\infty \int_{A_n^m}f\,dx\\
          &\leq \sum\limits_{n=1}^\infty \int_{A_n^m}(s_n^m-i_n^m)\,dx\\
          &<d^m\epsilon
    \end{align*}
    and
    \begin{align*}
          \int_{A_m}f\,dx-\sum\limits_{n=1}^\infty \int_{A_n^m}|i_n^m|\,dx&=\int_{A_m}f\,dx-\sum\limits_{n=1}^\infty \int_{A_n^m}i_n^m\,dx\\
          &=\sum\limits_{n=1}^\infty \int_{A_n^m}f\,dx-\sum\limits_{n=1}^\infty \int_{A_n^m}i_n^m\,dx\\
          &\leq \sum\limits_{n=1}^\infty \int_{A_n^m}(s_n^m-i_n^m)\,dx\\
          &<d^m\epsilon.
    \end{align*}
    Thus,
    \[
    \lim\limits_{m\to\infty}\sum\limits_{n=1}^\infty \int_{A_n^m}|s_n^m|\,dx=\lim\limits_{m\to\infty}\sum\limits_{n=1}^\infty \int_{A_n^m}|i_n^m|\,dx=0
    \]
    and hence both the series $\sum\limits_{n,m=1}^\infty \int_{A_n^m}|s_n^m|\,dx$ and $\sum\limits_{n,m=1}^\infty \int_{A_n^m}|i_n^m|\,dx$ converge in $\sR$. Therefore, ${\{i_n^m:A_n^m\to\sR\}_{n,m=1}^\infty}$, $\{s_n^m:A_n^m\to\sR\}_{n,m=1}^\infty$ are two collections of simple functions satisfying that $0\leq i_n^m\leq f\leq s_n^m$, the series $\sum\limits_{n,m=1}^\infty \int_{A_n^m}s_n^m\,dx$ and $\sum\limits_{n,m=1}^\infty \int_{A_n^m}i_n^m\,dx$ both converge in $\sR$ and
    \[
    \sum\limits_{n,m=1}^\infty \int_{A_n^m}(s_n^m-i_n^m)\,dx=\sum\limits_{m=1}^\infty\sum\limits_{n=1}^\infty \int_{A_n^m}(s_n^m-i_n^m)\,dx<\sum\limits_{m=1}^\infty d^m\epsilon<\epsilon
    \]
    which proves the result.
\end{proof}
\begin{col}
    Let $A\subseteq \sR$ and $f:A\to\sR$ be measurable. Then, $|f|$ is measurable.
\end{col}
\begin{proof}
    Let $\{A_n\}_{n=1}^\infty$ be a partition of $A$ and $\{i_n:A_n\to\sR\}_{n=1}^\infty$, $\{s_n:A_n\to\sR\}_{n=1}^\infty$ be two collections of simple functions such that $ i_n\leq f\leq s_n$, the series $\sum\limits_{n=1}^\infty \int_{A_n}|s_n|\,dx$ and $\sum\limits_{n=1}^\infty \int_{A_n}|i_n|\,dx$ both converge in $\sR$ and
    \[
    \sum\limits_{n=1}^\infty \int_{A_n}(s_n-i_n)\,dx<1.
    \]
    Since $f$ is bounded on each $A_n$, by \cref{bounded-abs-is-measurable}, $|f|$ is measurable on each $A_n$. We note that $|f|\leq|i_n|+|s_n|$. Thus,
    \[
    \int_{A_n}|f|\,dx\leq \int_{A_n}\left(|i_n|+|s_n|\right)\,dx
    \]
    and hence the series $\sum\limits_{m=1}^\infty \int_{A_m}|f|\,dx$ converges. It follows from \cref{partition-positive} that $|f|$ is measurable on $A$.
\end{proof}
\begin{col}
    Let $A\subseteq \sR$ be measurable and let $f,g:A\to\sR$ be measurable. Then, the functions $\min\{f,g\}$ and $\max\{f,g\}$ are measurable.
\end{col}
\begin{col}
    Let $A\subseteq \sR$ be measurable and let $f:A\to\sR$ be a function. Then, $f$ is measurable if and only if $f_+:=\max\{f,0\}$ and $f_-:=\max\{-f,0\}$ are measurable.
\end{col}
\begin{theorem}\label{measurable-criteria}
     Let $A\subseteq\sR$ be measurable, and let $f:A\to\sR$ be a function. Then, $f$ is measurable on $A$ if and only if there exists some partition $\{A_m\}_{m=1}^\infty$ of $A$ such that $f$ is measurable and bounded on each $A_m$ and the series
    \[
    \sum\limits_{m=1}^\infty \int_{A_m}|f|\,dx
    \]
    converges. Furthermore,
    \[
    \int_Af\,dx=\sum\limits_{m=1}^\infty \int_{A_m}f\,dx.
    \]
\end{theorem}
\begin{col}
    Let $A,B\subseteq \sR$ be measurable and let $f:A\cup B\to\sR$ be a function. Then, $f$ is measurable on $A\cup B$ if and only if it is measurable on $A$ and $B$.
\end{col}
\begin{propn}\label{aux5}
    Let $A\subseteq \sR$ be measurable and let $f,g:A\to\sR$ be measurable with $g$ bounded on $A$. Then $f\cdot g$ is measurable.
\end{propn}
\begin{proof} Without loss of generality, let $f,g\geq0$.

    Now, let $M>0$ in $\sR$ be a bound for $g$ on $A$ and let $\{A_n\}_{n=1}^\infty$ be a partition of $A$ such that $f$ is bounded on each $A_n$. We note that $f\cdot g$ is bounded on each $A_n$ and that
    \[
    \int_{A_n}f\cdot g\,dx\leq M \int_{A_n}f\,dx.
    \]
    By \cref{w-additivity}, the series $\sum\limits_{n=1}^\infty\int_{A_n}f\,dx$ converges. Thus, the series $\sum\limits_{n=1}^\infty\int_{A_n}f\cdot g\,dx$ converges too. It follows from \cref{partition-positive} that $f\cdot g$ is measurable.
\end{proof}
\begin{rmk}
    $f\cdot g$ need not be measurable if neither $f$ nor $g$ are bounded. Take for example the set $A:=\bigcup\limits_{n=1}^\infty (d^{2n},2d^{2n})$ and the function $f:A\to\sR$ given by $f(x):=\dfrac{1}{d^n}$ for $x\in(d^{2n},2d^{2n})$. Then, $f$ is measurable on $A$, but $f^2$ is not.
\end{rmk}
\begin{col}
    Let $A\subseteq \sR$ and $f:A\to\sR$ be measurable and let $B\subseteq A$ be a measurable subset of $A$. Then
    \[
    \int_A f\cdot\chi_B\,dx=\int_B f\,dx.
    \]
\end{col}
\subsection{Sets of measure zero}
We now explore some properties that this integral satisfies involving sets of measure zero.
\begin{ntn}
    We say that a property $P(x)$ holds almost everywhere in $A\subseteq\sR$, and write "$P$ holds a.e. in $A$", if there exists a measurable subset $U\subseteq A$ such that $m(U)=0$ and $P(x)$ holds true for $x\in A\bkl U$.
\end{ntn}
\begin{propn}
    Let $A\subseteq \sR$ be measurable with $m(A)=0$ and let $f:A\to\sR$ be a function. Then $f$ is measurable on $A$ and
    \[
    \int_A f\,dx=0.
    \]
\end{propn}
\begin{proof}
    Since the measure in $\sR$ is complete, the sets $A_n:=f^{-1}\left(-\frac{1}{d^n},\frac{1}{d^n}\right)$ are measurable and their union is $A$. Thus, the collection $\{B_n\}$ given by $B_1:=A_1$ and $B_{n+1}:=A_{n+1}\bkl A_n$ is a partition of $A$ and $f$ is bounded on each $B_n$. The result then follows from \cref{measurable-criteria}.
\end{proof}
\begin{propn}
     Let $A\subseteq \sR$ be measurable and let $f,g:A\to\sR$ be functions such that $f=g$ a.e. in $A$. Then, $f$ is measurable on $A$ if and only if $g$ is measurable on $A$, in which case, we have:
     \[
     \int_A f\,dx=\int_A g\,dx.
     \]
\end{propn}
\begin{proof}
    Let $U\subseteq A$ be a set of measure zero such that for $x\in A\bkl U$, $f(x)=g(x)$. Since both $f$ and $g$ are measurable on $U$, then $f$ and $g$ are measurable on $A$ if and only if they are measurable on $A\bkl U$. Given that $f=g$ on $A\bkl U$, it follows that $f$ is measurable on $A$ if and only if $g$ is measurable on $A$. The equality then follows from
    \[
    \int_A f\,dx=\int_U f\,dx+\int_{A\bkl U} f\,dx=\int_{A\bkl U} f\,dx=\int_{A\bkl U} g\,dx=\int_U g\,dx+\int_{A\bkl U} g\,dx=\int_A g\,dx.
    \]
\end{proof}
\begin{theorem}
    Let $A\subseteq \sR$ and $f:A\to\sR$ be measurable with $f\geq 0$. Then,
    \[
    \int_Af\,dx=0
    \]
    if and only if $f=0$ a.e. in $A$.
\end{theorem}
\begin{proof}
    Let's suppose that $\int_Af\,dx=0$, let $\epsilon>0$ in $\sR$ be given and let $(\{A_n^k\}_{n=1}^\infty)_{k=1}^\infty$ be a sequence of partitions of $A$ and ${(\{s_n^k:A_n^k\to\sR\}_{n=1}^\infty)_{k=1}^\infty}$ be a sequence of collections of simple functions such that $0\leq f\leq s^k_n$, the series $\sum\limits_{n=1}^\infty \int_{A_n^k}|s_n^k|\,dx$ converges in $\sR$ and
    \[
    \sum\limits_{n=1}^\infty \int_{A_n^k}s_n^k\,dx<d^{2k}\epsilon.
    \]
    We define for $k\in\N$
    \[
    U_k:=\{x\in A_n^k\mid s_n^k(x)\geq d^k,\, n\in\N\}.
    \]
    Using the same arguments as in \cref{char}, the sets $U_k$ are measurable. We note that
    \[
    m(U_k)= \dfrac{1}{d^k}\sum\limits_{n=1}^\infty \int_{A_n^k\cap U_k}d^k\,dx\leq  \dfrac{1}{d^k}\sum\limits_{n=1}^\infty \int_{A_n^k\cap U_k}s_n^k \,dx\leq\dfrac{1}{d^k}\sum\limits_{n=1}^\infty \int_{A_n^k}s_n^k\,dx<d^{k}\epsilon,
    \]
    thus $U:=\bigcup\limits_{k=1}^\infty U_k$ is measurable and has measure $m(U)<\epsilon$. Since for $x\in A\bkl U$, $f(x)=0$, the result then follows.
\end{proof}
\begin{rmk}
    The analogous for \cref{aux1}, \cref{aux2}, \cref{aux3}, \cref{aux4}, \cref{aux5}, \cref{UC1} and \cref{UC2} hold true whenever the hypothesis hold true a.e. in $\sR$.
\end{rmk}
\subsection{Classical theorems}
We now explore which classical theorems hold for this integral.
\begin{theorem}[Fundamental theorem of calculus]
    Let $f:[a,b]\to\sR$ be a measurable function that is continuous at $c\in[a,b]$. Then, the function $F:[a,b]\to\sR$ given by
    \[
    F(x):=\int\limits_a^xf(t)\,dt
    \]
    is differentiable at $c$ and has derivative $F'(c)=f(c)$.
\end{theorem}
\begin{proof}
    Let $\epsilon>0$ in $\sR$ be given and let $\delta>0$ in $\sR$ be such that if $x\in (c-\delta,c+\delta)$, then $|f(x)-f(c)|<\epsilon$. Now, let $h\in(-\delta,\delta)\bkl\{0\}$. Without loss of generality, let $h>0$. It follows that
    \begin{align*}
        \left|\dfrac{1}{h}[F(c+h)-F(c)]-f(c)\right|&=\left|\dfrac{1}{h}\int\limits_c^{c+h}[f(x)-f(c)]\,dx\right|\\
        &\leq\dfrac{1}{h}\int\limits_c^{c+h}\epsilon\,dx\\
        &=\epsilon,
    \end{align*}
    thus, $F$ is differentiable at $c$ and has derivative $F'(c)=f(c)$.
\end{proof}
\begin{rmk}
    If $f$ is bounded on $[a,b]$ and not necessarily continuous, then, $F$ is uniformly continuous on $[a,b]$.
\end{rmk}
\begin{theorem}[Uniform convergence]\label{UC1}
    Let $(f_n:A\to\sR)_{n\in\N}$ be a sequence of measurable functions that converges uniformly to a function $f:A\to\sR$. Then, $f$ is measurable on $A$. Furthermore,
    \[
    \int_Af\,dx=\lim\limits_{n\to\infty}\int_Af_n\,dx.
    \]
\end{theorem}
\begin{proof}
    Without loss of generality, let's suppose that $m(A)>0$.\\

    Let $\epsilon>0$ in $\sR$ be given and let $N\in \N$ be such that
    \[
    f_N-\dfrac{d\epsilon}{m(A)}<f<f_N+\dfrac{d\epsilon}{m(A)}.
    \]
    By definition, there exists some partition $\{A_n\}$ of $A$ and two collections of simple functions $\{i_n:A_n\to\sR\}$ and $\{s_n:A_n\to\sR\}$ such that $i_n\leq f_N\leq s_n$, the series $\sum\limits_{n=1}^\infty \int_{A_n}|s_n|\,dx$ and $\sum\limits_{n=1}^\infty \int_{A_n}|i_n|\,dx$ both converge in $\sR$ and
    \[
    \sum\limits_{n=1}^\infty \int_{A_n}(s_n-i_n)\,dx<d\epsilon.
    \]
    Then, the two collections of simple functions
    \[
    \left\{i_n-\dfrac{d\epsilon}{m(A)}:A_n\to\sR\right\}
    \]
    and
    \[
    \left\{s_n+\dfrac{d\epsilon}{m(A)}:A_n\to\sR\right\}
    \]
    satisfy that
    \[
    i_n-\dfrac{d\epsilon}{m(A)}\leq f\leq s_n+\dfrac{d\epsilon}{m(A)},
    \]
    the series
    \[
    \sum\limits_{n=1}^\infty \int_{A_n}\left|s_n+\dfrac{d\epsilon}{m(A)}\right|\,dx
    \text{ and }\sum\limits_{n=1}^\infty \int_{A_n}\left|i_n-\dfrac{d\epsilon}{m(A)}\right|\,dx
    \]
    both converge in $\sR$ and
    \begin{align*}
    \sum\limits_{n=1}^\infty \int_{A_n}\left(s_n+\dfrac{d\epsilon}{m(A)}-i_n+\dfrac{d\epsilon}{m(A)}\right)\,dx&= \sum\limits_{n=1}^\infty \int_{A_n}\left(s_n-i_n\right)\,dx+ \sum\limits_{n=1}^\infty \int_{A_n}\dfrac{2d\epsilon}{m(A)}\,dx\\
    &=\sum\limits_{n=1}^\infty \int_{A_n}\left(s_n-i_n\right)\,dx+ 2d\epsilon\\
    &<3d\epsilon\\
    &<\epsilon.
    \end{align*}
    To show that
    \[
    \int_Af\,dx=\lim\limits_{n\to\infty}\int_Af_n\,dx,
    \]
    we note that $|f-f_n|<U(n)$ where $U(n)\underset{n\to\infty}{\to}0$. Thus,
    \[
    \left|\int_Af\,dx-\int_Af_n\,dx\right|=\left|\int_A(f-f_n)\,dx\right|\leq U(n)\cdot m(A).
    \]
    Taking the limit as $n\to\infty$ proves the result.
\end{proof}
\begin{col}\label{UC2}
    Let $(f_n:A\to\sR)_{n\in\N}$ be a sequence of measurable functions that converges uniformly to $0$. Then, $f:=\sum\limits_{n=1}^\infty f_n$ is measurable on $\sR$. Furthermore,
    \[
    \int_Af\,dx=\sum\limits_{n=1}^\infty\int_Af_n\,dx.
    \]
\end{col}
\begin{rmk}
    As for the analogous of the monotone and dominated convergence theorems, we provide the following counter-example. Consider the sequence of functions $(f_n)_{n=1}^\infty$ given by $f_n:=\chi_{(0,d^{1/n})}+\chi_{(1/n,1)}$. Then, the sequence converges monotonously to and is dominated by $f:=\chi_{(0,1)}$, but the sequence of the integrals $\int_{(0,1)}f_n\,dx=1+d^{1/n}-1/n$ diverges in $\sR$. It's still an open question as to whether or not the addition of the hypothesis that the sequence of integrals $\left(\int_{A}f_n\,dx\right)_{n=1}^\infty$ converges in $\sR$ is enough to prove a weaker version of the monotone convergence theorem. The additional hypothesis does not, however, suffice to prove a weaker version of the dominated convergence theorem. For that, we provide the following counter-example:
\end{rmk}
\begin{eg}
    Consider the sequence of measurable functions $(f_n:(0,1)\to\sR)_{n=1}^\infty$ given by
    \[
    f_n:=\chi_{(0,d^{1/n})}+\chi_{(1/n,1)}-\chi_{(d^{1/n},2d^{1/n})}+\frac{1}{nd^{1/n}}\chi_{(2d^{1/n},3d^{1/n})}+
    \frac{d}{d^{1/n}}\chi_{(3d^{1/n},4d^{1/n})}.
    \]
    Since $d^{1/n}\gg d$, then the functions $f_n$ are dominated by $g=d^{-2}$ which is measurable on $(0,1)$. It's clear that the functions $f_n$ converge point-wise to the constant function $f=1$ and a simple calculation shows that $\int_0^1f_n\,dx=1+d$. Thus
    \[
    \lim_{n\to\infty}\int_0^1f_n\,dx=1+d\neq1=\int_0^1\lim_{n\to\infty}f_n\,dx.
    \]
    A similar argument shows that the sequence of functions might have been chosen such that the limit is equal to any other $\alpha\in\sR$.
\end{eg}
\section{Future work}\label{S6}
In the same way that we managed to generalize the 1-dimensional measure to the Levi-Civita space $\sR^j$ as shown in \cite{LMeasure24}, we hope that by generalizing the notion of a $\sR$-analytic function $f$ on an interval $[a,b]$ of $\sR$ to an $\sR$-analytic function whose domain is a simplex $S$ of $\sR^j$ in such a way that the notion for the integral of $f$ over $S$ makes sense, we will be able to generalize the one dimensional integral to an arbitrary $n$-dimensional case.

Also, given that if $x_0+a_1x_1\ldots+a_nx_n-1=0$ is the equation of an $n$-dimensional affine subspace of $\sR^{n+1}$ such that $a_i>0$ for $i=1,\ldots,n$, then, the volume of the simplex defined as the region spanned by all the positive axes that's bound by the subspace, as defined in \cite{TroncosoBerz13}, coincides with the appropriate iterated integrals (done in any order) of the function $f(x_1,\ldots,x_n)=1-(a_1x_1+\ldots+a_nx_n)$, motivates us to reformulate the theory presented in \cite{int2-3} in terms of simple regions. Given that, for example, the unit ball $\mathbb{D}^2=\{(x,y)\in\sR^2:x^2+y^2\leq1\}$ is not measurable with the current theory, but with the new approach of this paper it will be possible to assign it the measure $m(\mathbb{D}^2)=\pi$.
\section*{Funding}
K. Shamseddine’s research is funded by the Natural Sciences and Engineering Research Council of Canada (NSERC, Grant \#RGPIN-2023-04315).
\section*{Conflict of Interest}
The authors of this work declare that they have no conflicts of interest.
\bibliographystyle{plaindin}
\bibliography{refs}
\end{document}